\documentclass[10pt]{amsart}
\usepackage[nohug,heads=littlevee]{diagrams}
\usepackage{mathrsfs}
\usepackage{bm}
\usepackage{amssymb}
\usepackage[latin1]{inputenc}
\usepackage{xr-hyper}
\usepackage[plainpages=false,pdfpagelabels]{hyperref}
\usepackage[left=3.5cm,top=2.5cm,right=3.6cm,bottom=1.4in,asymmetric]{geometry}
\usepackage{mathtools}
\usepackage{enumitem}
\setenumerate[1]{label=(\thesubsection.\arabic*)}

\usepackage[backend=bibtex,style=alphabetic,maxbibnames=99,maxcitenames=99,maxalphanames=99,backref=false,sorting=nyt]{biblatex}
\renewbibmacro{in:}{%
  \ifentrytype{article}{}{\printtext{\bibstring{in}\intitlepunct}}}
 \DeclareFieldFormat{postnote}{#1}
 \DeclareFieldFormat{multipostnote}{#1}

\usepackage{color}

\newcommand{\defnword}[1]{\textbf{#1}}
\newcommand{\comment}[1]{}
\newcommand{\on}[1]{\operatorname{#1}}

\numberwithin{equation}{subsection}
\newtheorem{introthm}{Theorem}

\newtheorem{proposition}[subsection]{Proposition}
\newtheorem{theorem}[subsection]{Theorem}
\newtheorem*{thm*}{Theorem}
\newtheorem{lemma}[subsection]{Lemma}
\newtheorem*{lem*}{Lemma}

\theoremstyle{definition}
\newtheorem{defn}[subsection]{Definition}

\theoremstyle{remark}

\newtheorem{remark}[subsection]{Remark}

\newtheorem*{assump*}{Assumption}

\newarrow{Equals}{=}{=}{}{=}{}

\newcommand{\R}{\mathbb{R}}
\newcommand{\Z}{\mathbb{Z}}

\newcommand{\C}{\mathbb{C}}
\newcommand{\A}{\mathbb{A}}
\newcommand{\F}{\mathbb{F}}
\newcommand{\Q}{\mathbb{Q}}
\newcommand{\Lie}{\on{Lie}}

\newcommand{\coker}{\on{coker}}

\newcommand{\Rg}{\mathscr{O}}
\newcommand{\Reg}[1]{\Rg_{#1}}
\newcommand{\co}{\mathcal{O}}

\newcommand{\Hom}{\on{Hom}}
\newcommand{\End}{\on{End}}

\newcommand{\pow}[1]{[\vert #1|]}

\newcommand{\Gal}{\on{Gal}}

\newcommand{\Gm}{\mathbb{G}_m}

\newcommand{\Spec}{\on{Spec}}

\newcommand{\alg}{\mathrm{alg}}

\newcommand{\cris}{\on{cris}}
\newcommand{\dR}{\on{dR}}

\newcommand{\et}{\text{\'et}}
\newcommand{\Fil}{\on{Fil}}

\newcommand{\Sh}{\on{Sh}}

\newcommand{\GSp}{\on{GSp}}

\newcommand{\GL}{\on{GL}}

\newcommand{\SO}{\on{SO}}

\newcommand{\GSpin}{\on{GSpin}}

\newcommand{\op}{\on{op}}

\begin{document}
\title[$2$-adic integral canonical models]{$2$-adic integral canonical models and the Tate conjecture in characteristic $2$}
\author{Wansu Kim and Keerthi Madapusi Pera}

\address{Department of Mathematics, University of Chicago, 5734 S University Ave, Chicago, IL, USA}
\email{keerthi@math.uchicago.edu}

\address{Department of Mathematics, King's College London, Strand, London, WC2R 2LS, UK}
\email{wansu.kim@kcl.ac.uk}

\thanks{K.~Madapusi Pera is supported by NSF grant DMS-1502142.}

\begin{abstract}
We use E. Lau's classification of $2$-divisible groups using Dieudonn\'e displays to construct integral canonical models for Shimura varieties of abelian type at $2$-adic places where the level is hyperspecial. We apply this to prove the Tate conjecture for K3 surfaces in characteristic $2$.
\end{abstract}

\maketitle

\section*{Introduction}

In this note, we complete the construction of integral canonical models from~\cite{kisin:abelian}at places of hyperspecial level, so that it also works at $2$-adic places, without any additional restrictions. Therefore, we obtain the following theorem.

\begin{introthm}
\label{thm:main}
Let $\Sh_K(G,X)$ be a Shimura variety of abelian type associated with a Shimura datum $(G,X)$ and a neat level $K\subset G(\A_f)$, defined over the reflex field $E = E(G,X)$. Suppose that for a prime $p$ the $p$-primary part $K_p$ is hyperspecial. Then $\Sh_K(G,X)$ admits an integral canonical model over $\Reg{E,(v)}$ for any place $v\vert p$ of $E$.
\end{introthm}

The non-trivial input is Lau's classification of $p$-divisible groups over a very large class of $2$-adic rings in terms of Dieudonn\'e displays~\cite{lau:displays}, and its compatibility with a $p$-adic Hodge theoretic construction of Kisin~\cite{Lau2012-il}.

These results will find use in a joint project of the second author with F. Andreatta, E.-Z. Goren and B. Howard on an averaged version of the Colmez conjecture on heights of abelian varieties with complex multiplication~\cite{aghmp:colmez}. Moreover, they extend Kisin's proof of the Langlands-Rapoport conjecture~\cite{kisin:abelian} for Shimura varieties of abelian type, so that it works even at the $2$-adic places. 

As a more immediate application, we prove:
\begin{introthm}
\label{thm:k3}
The Tate conjecture holds for K3 surfaces over finitely generated fields.
\end{introthm}
Of course, the case when the field has characteristic $\neq 2$, this is already known by the results of the second author in~\cite{mp:tatek3}, and also by earlier work by Maulik~\cite{Maulik2014-kj} and Charles~\cite{Charles2013-bl}. The new ingredient here is the characteristic $2$ case.

A remark on notation: Given a map $f:R\to S$ of commutative rings, and a module $M$ over $R$, we will use the geometric notation $f^*M$ for the change of scalars $S\otimes_{f,R}M$. Given a ring $R$ and an object $\mathcal{M}$ in an $R$-linear rigid tensor category $\mathbf{C}$, we will write $\mathcal{M}^\otimes$ for the ind-object over $\mathbf{C}$ given by the direct sum of the tensor, exterior and symmetric powers of $\mathcal{M}$ and its dual. Given a scheme $X$ over a ring $R$ and a map $R\to R'$ of rings, we will write $X_{R'}$ or $R'\otimes_RX$ for the base change of $X$ over $R'$. Given a finite free $R$-module $M$, we will write $M^\otimes$ for the direct sum of all $R$-modules that can be formed via the operations of taking duals, tensor products, and symmetric and exterior powers of $M$.

For any other possibly unfamiliar notions, we refer the reader to the Appendix of~\cite{kisin:fcrystals} and \S 1.1 of~\cite{kisin:abelian}.

\section{Lau's classification and integral $p$-adic Hodge theory}\label{sec:lau_stuff}

Fix a perfect field $k$ in characteristic $p$. Set $W = W(k)$ and let $K_0 = \mathrm{Frac}(W)$ be its fraction field. Write $\sigma:W\to W$ for the canonical lift of the $p$-power Frobenius automorphism of $k$. Let $K/K_0$ be a finite totally ramified extension. Choose a uniformizer $\pi\in K$, and let $\mathcal{E}(u)\in W[u]$ be the associated Eisenstein polynomial with constant term $\mathcal{E}(0) = p$, so that $e = \deg \mathcal{E}(u)$ is the ramification index of $K$. Fix an algebraic closure $K_0^\alg$ of $K_0$, as well as an embedding $K\subset K_0^\alg$. Let $\Gamma_K = \Gal(K_0^\alg/K)$ be the absolute Galois group of $K$.

\subsection{}\label{subsec:mfS and S}
Let $\mathfrak{S} = W\pow{u}$ be the power series ring in one variable over $W$ and equip it with the Frobenius lift $\varphi:\mathfrak{S}\to \mathfrak{S}$ satisfying $\varphi(u) = u^p$. 

Let $S$ be the $p$-adic completion of the divided power envelope of the surjection $W[u]\to \co_K$ carrying $u$ to $\pi$. Explicitly, $S$ is the $p$-adic completion of the subring
\[
W\biggl[u,\frac{u^{ei}}{i!}:\;i\in\Z_{\geq 1}\biggr]\subset K_0[u].
\]
The natural map $W[u]\to S$ extends to an embedding $\mathfrak{S}\hookrightarrow S$, and the Frobenius lift $\varphi:\mathfrak{S}\to\mathfrak{S}$ extends continuously to an endomorphism $\varphi:S\to S$.

\subsection{}
A \emph{Breuil-Kisin module} over $\co_K$ (with respect to $\varpi$) is a pair $(\mathfrak{M},\varphi_{\mathfrak{M}})$, where $\mathfrak{M}$ is a finite free $\mathfrak{S}$-module and $\varphi_{\mathfrak{M}}:\varphi^*\mathfrak{M}[\mathcal{E}^{-1}]\xrightarrow{\simeq}\mathfrak{M}[\mathcal{E}^{-1}]$ is an isomorphism of $\mathfrak{S}$-modules. Here, $\varphi:\mathfrak{S}\to\mathfrak{S}$ is the Frobenius lift that extends the canonical Frobenius automorphism $\mathrm{Fr}:W\to W$ and satisfies $\varphi(u)=u^p$. 

Usually, the map $\varphi_{\mathfrak{M}}$ will be clear from context and we will denote the Breuil-Kisin module by its underlying $\mathfrak{S}$-module $\mathfrak{M}$.

For any integer $i$, we will write $\mathbf{1}(i)$ for the Breuil-Kisin module whose underlying $\mathfrak{S}$-module is just $\mathfrak{S}$ equipped with $\mathcal{E}(u)^{-i}$-times the canonical identification $\varphi^*\mathfrak{S}[\mathcal{E}(u)^{-1}] = \mathfrak{S}[\mathcal{E}(u)^{-1}]$. When $i=0$, we will write $\mathbf{1}$ instead of $\mathbf{1}(0)$. 

There are natural notions of tensor, exterior and symmetric products on the category of Breuil-Kisin modules. For any Breuil-Kisin module $\mathfrak{M}$, we will set $\mathfrak{M}(i) = \mathfrak{M}\otimes_{\mathfrak{S}}\mathbf{1}(i)$.

The \defnword{dual} $(\mathfrak{M}^\vee,\varphi_{\mathfrak{M}^\vee})$ of a Breuil-Kisin module $\mathfrak{M}$ is the dual $\mathfrak{S}$-module $\mathfrak{M}^\vee$ equipped with the isomorphism
\[
\varphi_{\mathfrak{M}^\vee}= \bigl(\varphi^\vee_{\mathfrak{M}}\bigr)^{-1}:\varphi^*\mathfrak{M}^\vee[\mathcal{E}(u)^{-1}]\xrightarrow{\simeq}\mathfrak{M}^\vee[\mathcal{E}(u)^{-1}],
\]
where $\varphi^\vee_{\mathfrak{M}}$ is the $\mathfrak{S}$-linear dual of $\varphi_{\mathfrak{M}}$.

\subsection{}\label{subsec:kisin_functor}
By~\cite[Theorem (1.2.1)]{kisin:abelian}, there is a (covariant) 
fully faithful tensor functor $\mathfrak{M}$ from the category of $\Z_p$-lattices in crystalline $\Gamma_K$-representations to the category of Breuil-Kisin modules over $\co_K$. It has various useful properties. To describe them, fix a crystalline $\Z_p$-representation $\Lambda$. Then:

\begin{itemize}
\item If $\Lambda = \Z_p(i)$ is the rank $1$-representation of $\Gamma_K$ attached to the $i^{\mathrm{th}}$-power of the $p$-adic cyclotomic character $\chi_p:\Gamma_K\to \Z_p^\times$, then there is a natural identification
\begin{equation}\label{bk:trivial}
\mathfrak{M}(\Z_p(i)) = \mathbf{1}(i).
\end{equation}

\item There is a canonical isomorphism of Breuil-Kisin modules:
\begin{equation}\label{bk:duality}
\mathfrak{M}(\Lambda^\vee)\xrightarrow{\simeq} \mathfrak{M}(\Lambda)^\vee,
\end{equation}
where $\Lambda^\vee:=\Hom_{\Z_p}(\Lambda,\Z_p)$. 

\item For any $i\in\Z_{\geq 0}$, there are canonical isomorphisms:
\begin{align}\label{bk:symm_ext}
\mathfrak{M}(\mathrm{Sym}^i\Lambda) \xrightarrow{\simeq} \mathrm{Sym}^i\mathfrak{M}(\Lambda);\\
\mathfrak{M}(\wedge^i\Lambda) \xrightarrow{\simeq} \wedge^i\mathfrak{M}(\Lambda)\nonumber
\end{align}
of Breuil-Kisin modules.

\item There is a canonical isomorphism of $F$-isocrystals over $\mathrm{Frac}(W)$:
	\begin{equation}\label{bk:dcris}
     \varphi^*\mathfrak{M}(\Lambda)/u\varphi^*\mathfrak{M}(\Lambda)[p^{-1}]\xrightarrow{\simeq}D_{\cris}(\Lambda)=(\Lambda\otimes_{\Z_p}B_{\cris})^{\Gamma_K}.
	\end{equation}
	
\item Equip $\varphi^*\mathfrak{M}(\Lambda)$ with the descending filtration $\Fil^\bullet\mathfrak{M}$ given by:
\[
\Fil^i\varphi^*\mathfrak{M}(\Lambda) = \{x\in\varphi^*\mathfrak{M}(\Lambda)\varphi_{\mathfrak{M}(\Lambda)}(x)\in\mathcal{E}(u)^i\mathfrak{M}(\Lambda)\}.
\] 
Then there is a canonical isomorphism of filtered $E$-vector spaces
\begin{equation}\label{bk:ddr}
(\varphi^*\mathfrak{M}(\Lambda)/\mathcal{E}(u)\varphi^*\mathfrak{M}(\Lambda))[p^{-1}]\xrightarrow{\simeq}K\otimes_{\mathrm{Frac}(W)}D_{\cris}(\Lambda) = D_{\dR}(\Lambda).
\end{equation}
Here, the left hand side is equipped with the filtration induced from $\Fil^\bullet\varphi^*\mathfrak{M}(\Lambda)$.  

\item The functor $\mathfrak{M}$ is compatible with unramified base change: If $k'/k$ is a finite extension, set $K'=K\otimes_WW(k')$, and fix an embedding $K'\hookrightarrow K_0^{\alg}$. If $\mathfrak{S}_{k'} = W(k')\pow{u}$, we obtain a functor $\mathfrak{M}_{k'}$ from $\Z_p$-lattices in crystalline representations of $\Gamma_{K'}=\Gal(K_0^{\alg}/K')$ to Breuil-Kisin modules over $\co_{K'}$ consisting of pairs $(\mathfrak{M}',\varphi_{\mathfrak{M}'})$ with $\mathfrak{M}'$ finite free over $\mathfrak{S}_{k'}$. For any crystalline $\Z_p$-representation $\Lambda$ of $\Gamma_K$, we now have a canonical isomorphism
\begin{equation}
\label{bk:unram base change}
\mathfrak{M}_{k'}(\Lambda\vert_{\Gamma_{K'}})\xrightarrow{\simeq}\mathfrak{S}_{k'}\otimes_{\mathfrak{S}}\mathfrak{M}(\Lambda)
\end{equation}
of Breuil-Kisin modules.

\end{itemize}

\subsection{}
A crucial property of the functor $\mathfrak{M}$ is what Kisin calls `the Key Lemma'. To explain this, let $\Lambda$ be as above. and let 
\[
\mathfrak{M} \coloneqq \mathfrak{M}(\Lambda)
\]
 be its associated Breuil-Kisin module. Suppose that we are given a collection of $\Gamma_K$-invariant tensors $\{s_{\alpha}\}\subset\Lambda^\otimes$, which we can view as $\Gamma_K$-equivariant maps
\[
s_{\alpha}:\Z_p \to \Lambda^\otimes.
 \]
 By~\eqref{bk:trivial},~\eqref{bk:duality} and~\eqref{bk:symm_ext}, these give rise to maps:
\[
s_{\alpha,\mathfrak{M}} : \mathbf{1} \to \mathfrak{M}^\otimes,
\]
which we can view as a collection of $\varphi$-invariant elements $\{s_{\alpha,\mathfrak{M}}\}\subset\mathfrak{M}^\otimes$. 

Set $M_{\cris} = \varphi^*\mathfrak{M}/u\varphi^*\mathfrak{M}$ and $M_{\dR} = \varphi^*\mathfrak{M}/\mathcal{E}(u)\varphi^*\mathfrak{M}$. Write $\Fil^\bullet M_{\dR}$ for the Hodge filtration on $M_{\dR}$: This is the filtration of $M_{\dR}$ of obtained by taking for each $i\in\Z$, the saturation of the image of $\Fil^i\varphi^*\mathfrak{M}$ in $M_{\dR}$. 

From $\{s_{\alpha,\mathfrak{M}}\}$, we obtain $\varphi$-invariant tensors $\{s_{\alpha,\cris}\}\subset M_{\cris}^\otimes$, as well as tensors $\{s_{\alpha,\dR}\}\subset \Fil^0M_{\dR}^\otimes$.

\begin{theorem}\label{thm:kisin_key_lemma}
Suppose that the pointwise stabilizer $G\subset\GL(\Lambda)$ of $\{s_{\alpha}\}$ is a connected reductive group over $\Z_p$. Assume that $k$ is either finite or algebraically closed. Then there is an isomorphism
\[
\mathfrak{S}\otimes_{\Z_p}\Lambda \xrightarrow{\simeq}\mathfrak{M}
\]
carrying $1\otimes s_{\alpha}$ to $s_{\alpha,\mathfrak{S}}$. for each $\alpha$. Therefore, the stabilizer in $\GL(M_{\cris})$ of $\{s_{\alpha,\cris}\}$ (resp. in $\GL(M_{\dR})$ of $\{s_{\alpha,\dR}\}$) is isomorphic to $G_{W}$ (resp. to $G_{\co_K}$). Moreover, the filtration $\Fil^\bullet M_{\dR}$ is split by a cocharacter $\mu:\mathbb{G}_{m,\co_K}\to G_{\co_K}$. 
\end{theorem}
\begin{proof}
This follows from Corollary 1.3.5 and the argument from Corollary 1.4.3 of~\cite{kisin:abelian}.
\end{proof}

\begin{remark}\label{rem:key lemma other way}
Although this is not explicitly stated in~\cite{kisin:abelian}, if we assume that the pointwise stabilizer $G_{\mathfrak{S}}\subset\GL(\mathfrak{M})$ is connected reductive, but we do \emph{not} assume that $G$ is so, then the proof of Proposition~1.3.4 of~\emph{loc. cit.} shows that we still have an isomorphism
\[
\mathfrak{S}\otimes_{\Z_p}\Lambda \xrightarrow{\simeq}\mathfrak{M}
\]
carrying $1\otimes s_{\alpha}$ to $s_{\alpha,\mathfrak{S}}$. for each $\alpha$: Indeed, the reductivity of $G$ is only used in Step 5 of the proof, and its use there can be replaced with that of the reductivity of $G_{\mathfrak{S}}$. In particular, since $\mathfrak{S}$ is faithfully flat over $\Z_p$, we conclude \emph{a posteriori} that $G$ is connected reductive over $\Z_p$.
\end{remark}

\subsection{}
Given an integer $a\in\Z_{\geq 0}$, a \defnword{Breuil window of level $a$} is a triple $(\mathfrak{P}_a,\Fil^1\mathfrak{P}_a,\tilde{\varphi}_{a})$, where: 

\begin{itemize}
	\item  $\mathfrak{P}_a$ is a finite free module over $\mathfrak{S}_a \coloneqq\mathfrak{S}/u^{a+1}\mathfrak{S}$;

	\item $\Fil^1\mathfrak{P}_a\subset \mathfrak{P}_a$ is a free $\mathfrak{S}_a$-submodule containing $\mathcal{E}(u)\mathfrak{P}_a$ such that $\mathfrak{P}_a/\Fil^1\mathfrak{P}_a$ is a finite free module over $\co_K/\pi^{a+1}$;

	\item $\tilde{\varphi}_a:\varphi^*\Fil^1\mathfrak{P}_a\xrightarrow{\simeq}\mathfrak{P}_a$ is an isomorphism of $\mathfrak{S}_a$-modules.
\end{itemize}

Breuil windows of level $a$ form an exact $\Z_p$-linear category for the obvious notion of morphism and short exact sequences.

\subsection{}\label{subsec:windows lau}
Breuil windows of level $a$ are a special case of a definition from~\cite[\S 2.1]{lau:displays}. Define a $\varphi$-semilinear map
\begin{align*}
\varphi_1: \mathcal{E}(u)\mathfrak{S}_a&\to \mathfrak{S}_a\\
\mathcal{E}(u)x&\mapsto \varphi(x).
\end{align*}

In the notation of \emph{loc. cit.}, the tuple 
\[
\mathscr{B}_a = (\mathfrak{S}_a,\mathcal{E}(u)\mathfrak{S}_a,\co_K/\pi^{a+1},\varphi,\varphi_1)
\]
is a \emph{lifting frame}. Lau considers the category of windows over $\mathscr{B}_a$. Windows are tuples $(P,Q,F,F_1)$, where $P$ is a free $\mathfrak{S}_a$-module, $Q\subset P$ is a $\mathfrak{S}_a$-submodule such that $P/Q$ is finite free over $\co_K/\pi^{a+1}$,\footnote{Note that, since $\mathcal{E}(u)$ is a non zero divisor, this implies that $Q$ is necessarily free over $\mathfrak{S}_a$.} $F:P\to P$ a $\varphi$-semilinear map, and $F_1:Q \to P$ is another $\varphi$-semilinear map satisfying
\[
F_1(\mathcal{E}(u)\cdot m) = F(m),
\]
for all $m\in Q$, and whose image generates $P$ as an $\mathfrak{S}_a$-module. Since $\mathcal{E}(u)$ is not a zero divisor, we see that $F$ is uniquely determined by $F_1$, and so the category of windows over $\mathscr{B}_a$ is equivalent to the category of triples $(P,Q,F_1)$, which is simply the category of Breuil windows of level $a$.

\subsection{}
Given a $p$-divisible group $\mathcal{H}$ over a $p$-adically complete ring $R$, we will consider the contravariant Dieudonn\'e $F$-crystal $\mathbb{D}(\mathcal{H})$ (see for instance~\cite{bbm:cris_ii}). 

Given any nilpotent thickening $R'\to R$, whose kernel is equipped with divided powers, we can evaluate $\mathbb{D}(\mathcal{H})$ on $R'$ to obtain a finite projective $R'$-module $\mathbb{D}(\mathcal{H})(R')$ (this construction depends on the choice of divided power structure, which will be specified or evident from context). If $R'$ admits a Frobenius lift $\varphi:R'\to R'$, then we get a canonical map
\[
\varphi:\varphi^*\mathbb{D}(\mathcal{H})(R')\to \mathbb{D}(\mathcal{H})(R')
\]
obtained from the $F$-crystal structure on $\mathbb{D}(\mathcal{H})$. 

An example of a (formal) divided power thickenings is any surjection of the form $R'\to R'/pR'$, where we equip $pR'$ with the canonical divided power structure induced from that on $p\Z_p$. Another example is the surjection $S\to \co_K$ from~\eqref{subsec:mfS and S}.

The evaluation on the trivial thickening $R\to R$ gives us a projective $R$-module $\mathbb{D}(\mathcal{H})(R)$ of finite rank equipped with a short exact sequence of projective $R$-modules:
\[
0 \to (\Lie \mathcal{H})^\vee\to \mathbb{D}(\mathcal{H})(R) \to \Lie\mathcal{H}^\vee\to 0,
\]
where $\mathcal{H}^\vee$ is the Cartier dual of $\mathcal{H}$.

We will set
\[
\Fil^1\mathbb{D}(\mathcal{H})(R)\coloneqq (\Lie \mathcal{H})^\vee\subset \mathbb{D}(\mathcal{H})(R).
\]
The associated descending $2$-step filtration $\Fil^\bullet\mathbb{D}(\mathcal{H})(R)$ concentrated in degrees $0$ and $1$ will be called the \defnword{Hodge filtration}. Sometimes, we will abuse terminology and refer to the summand $\Fil^1\mathbb{D}(\mathcal{H})(R)$ itself as the Hodge filtration.

\subsection{}
We will say that a Breuil-Kisin module $\mathfrak{M}$ has \emph{$\mathcal{E}$-height $1$} if the isomorphism $\varphi_{\mathfrak{M}}$ arises from a map $\varphi^*\mathfrak{M}\to\mathfrak{M}$ whose cokernel is killed by $\mathcal{E}(u)$. Write $\mathrm{BT}_{/\mathfrak{S}}$ for the category of Breuil-Kisin modules of $\mathcal{E}$-height $1$.

Note that, for any $p$-divisible group $\mathcal{H}$ over $\co_K$, the Breuil-Kisin module $\mathfrak{M}(T_p(\mathcal{H})^\vee)$ has $\mathcal{E}$-height $1$. Here, $\mathfrak{M}$ is the functor from~\eqref{subsec:kisin_functor}.

With each Breuil-Kisin module $\mathfrak{M}$, we can functorially associate a triple $(\mathfrak{P},\Fil^1\mathfrak{P},\tilde{\varphi}_{\mathfrak{P}})$, where $\mathfrak{P} = \varphi^*\mathfrak{M}$; $\Fil^1 \mathfrak{P}	\subset \mathfrak{P}$ is $\mathfrak{M}$, viewed as a submodule of $\mathfrak{P}$ via the unique map $V_{\mathfrak{M}}:\mathfrak{M}\to \varphi^*\mathfrak{M}$ whose composition with $\varphi_{\mathfrak{M}}$ is multiplication by $\mathcal{E}(u)$; and $\tilde{\varphi}_{\mathfrak{P}}:\varphi^*\Fil^1 \mathfrak{P}\xrightarrow{\simeq}\mathfrak{P}$ is the obvious isomorphism.

Observe that $\coker\tilde{\varphi}_{\mathfrak{P}}$ is free over $\co_K = \mathfrak{S}/\mathcal{E}(u)\mathfrak{S}$, and that $\mathcal{E}(u)$ is a non zero divisor in $\mathfrak{S}_a$. From this, it follows that the natural map
\[
\Fil^1 \mathfrak{P}\otimes_{\mathfrak{S}}\mathfrak{S}_a\to \mathfrak{P}\otimes_{\mathfrak{S}}\mathfrak{S}_a
\]
of finite free $\mathfrak{S}_a$-modules is injective, with cokernel finite free over $\co_K/\pi^{a+1}$.

Therefore, reducing the triple $(\mathfrak{P},\Fil^1\mathfrak{P},\tilde{\varphi}_{\mathfrak{P}})$ mod $u^{a+1}$ gives us a canonical functor
\[
\mathcal{B}^{\co_K}_{\co_K/\pi^{a+1}}: \mathrm{BT}_{/\mathfrak{S}} \to \mathrm{BT}_{/\mathfrak{S}_a}.
\]

Similarly, for any $a,b\in \Z_{\geq 0}$ with $b\geq a$, we also have a canonical reduction functor
\[
\mathcal{B}^{\co_K/\pi^{b+1}}_{\co_K/\pi^{a+1}}: \mathrm{BT}_{/\mathfrak{S}_b} \to \mathrm{BT}_{/\mathfrak{S}_a}
\]
obtained by reducing triples in $\mathrm{BT}_{/\mathfrak{S}_b}$ modulo $u^{a+1}$.

\subsection{}

For any commutative ring $R$, write $(p\mathrm{-div})_{R}$ for the category of $p$-divisible groups over $R$. Given a map $R\to R'$ of commutative rings, Write $\mathcal{B}^R_{R'}$ for the base change functor from $(p\mathrm{-div})_{R}$ to $(p\mathrm{-div})_{R'}$.

\begin{theorem}\label{thm:kisin_lau_p_divisible}
There are exact anti-equivalences of categories
\begin{align*}
\mathfrak{M}:\;(p\mathrm{-div})_{\co_K} &\xrightarrow{\simeq} \mathrm{BT}_{/\mathfrak{S}};\\
\mathfrak{P}_a:\; (p\mathrm{-div})_{\co_K/\pi^{a+1}}&\xrightarrow{\simeq} \mathrm{BT}_{/\mathfrak{S}_a}
\end{align*}
with the following properties:
\begin{enumerate}
    \item\label{kisin_lau:reduction}For each $a\in \Z_{\geq 0}$, we have a canonical isomorphism of functors
    \begin{align*}
     \mathfrak{P}_a\circ \mathcal{B}^{\co_K}_{\co_K/\pi^{a+1}}\xrightarrow{\simeq} \mathcal{B}^{\co_K}_{\co_K/\pi^{a+1}}\circ \mathfrak{M}
    \end{align*}
    from $(p\mathrm{-div})_{\co_K}$ to $\mathrm{BT}_{/\mathfrak{S}_a}$, and if $b\geq a$, we have a canonical isomorphism of functors
     \begin{align*}
     \mathfrak{P}_a\circ \mathcal{B}^{\co_K/\pi^{b+1}}_{\co_K/\pi^{a+1}}\xrightarrow{\simeq}\mathcal{B}^{\co_K/\pi^{b+1}}_{\co_K/\pi^{a+1}}\circ\mathfrak{P}_b
    \end{align*}
    from $(p\mathrm{-div})_{\co_K/\pi^{b+1}}$ to $\mathrm{BT}_{/\mathfrak{S}_a}$.

    \item\label{kisin_lau:compatibility}For each $p$-divisible group $\mathcal{H}$ over $\co_K$ there is a canonical isomorphism
    \[
     \mathfrak{M}(\mathcal{H}) \xrightarrow{\simeq} \mathfrak{M}(T_p(\mathcal{H})^\vee)
    \]
    in $\mathrm{BT}_{/\mathfrak{S}}$. 

	\item\label{kisin:cris}The $\varphi$-equivariant composition
\[
    \varphi^*\mathfrak{M}(\mathcal{H})/u\varphi^*\mathfrak{M}(\mathcal{H})  
    \xrightarrow[\simeq]{\eqref{bk:dcris}}	D_{\cris}(T_p(\mathcal{H})^\vee)	\xrightarrow{\simeq}	\mathbb{D}(\mathcal{H})(W)[p^{-1}]
\]
	maps $\varphi^*\mathfrak{M}(\mathcal{H})/u\varphi^*\mathfrak{M}(\mathcal{H})$ isomorphically onto $\mathbb{D}(\mathcal{H})(W)$. 
	\item\label{kisin:derham}The filtered isomorphism
	\begin{align*}
    \varphi^*\mathfrak{M}(\mathcal{H})/\mathcal{E}(u)\varphi^*\mathfrak{M}(\mathcal{H})[p^{-1}]&\xrightarrow[\simeq]{\eqref{bk:ddr}}K\otimes_{\mathrm{Frac}(W)}D_{\cris}(T_p(\mathcal{H})^\vee)\\
    &\xrightarrow{\simeq}\mathbb{D}(\mathcal{H})(\co_K)[p^{-1}]
	\end{align*}
	maps $\varphi^*\mathfrak{M}(\mathcal{H})/\mathcal{E}(u)\varphi^*\mathfrak{M}(\mathcal{H})$ isomorphically onto $\mathbb{D}(\mathcal{H})(\co_K)$, and hence maps $\Fil^1\mathfrak{M}(\mathcal{H})$ onto $\Fil^1 \mathbb{D}(\mathcal{H})(\co_K)$.
	\item\label{kisin:breuil}There is a canonical $\varphi$-equivariant isomorphism
	\[
    S\otimes_{\varphi,\mathfrak{S}}\mathfrak{M}(\mathcal{H})\xrightarrow{\simeq}\mathbb{D}(\mathcal{H})(S)
	\] 
	whose reduction along the map $S\to\co_K$ gives the filtration preserving isomorphism in~\ref{kisin:derham}.
\end{enumerate}
\end{theorem}
\begin{proof}
When $p>2$ most of this, except for the existence of the equivalences $\mathfrak{P}_a$, follows from (1.4.2) and (1.4.3) of~\cite{kisin:abelian}. When $p=2$, again, most of this follows~\cite{kim:2-adic}. However, the results of Lau from~\cite{lau:displays} lead to a uniform proof in all cases, and, more importantly, also at the same time give us the functors $\mathfrak{P}_a$ classifying $p$-divisible groups over the artinian rings $\co_K/\pi^{a+1}$.

To begin, the existence of the functors $\mathfrak{M}$ and $\mathfrak{P}_a$, as well as the compatibility between them asserted in~\ref{kisin_lau:reduction} follows from Theorem 6.6 and Corollary 5.4 of~\cite{lau:displays}\footnote{Our functors are the Cartier duals of the exact equivalences of categories defined in \cite{lau:displays}.}. Assertion~\ref{kisin:breuil} is Proposition 7.1 of \emph{loc. cit.}, and assertion~\ref{kisin_lau:compatibility} follows from the main result of~\cite{Lau2012-il}. Note that this corrects the claim in~\cite[Theorem (1.4.2)]{kisin:abelian}, which is off by a Tate twist; see~\cite[Theorem (1.1.6)]{kisin:lr}.

The remaining assertions now follow from the properties of the functor $\mathfrak{M}$.

\end{proof}



\section{Deformation theory}\label{sec:deformation}

Suppose now that $k$ is a finite field. Let $\mathcal{G}_0$ be a $p$-divisible group over $k$, and let $M_0 = \mathbb{D}(\mathcal{G}_0)(W)$ be the Dieudonn\'e $F$-crystal associated with it. Suppose that we have $\varphi$-invariant tensors $\{\bm{s}_{\alpha,0}\}\subset M_0^\otimes$, whose stabilizer is a connected reductive subgroup $G\subset \GL(M_0)$, and such that the Hodge filtration 
\[
\Fil^1(M_0\otimes k) \subset M_0\otimes k  = \mathbb{D}(\mathcal{G}_0)(k)
\]
is split by a cocharacter $\mu_0:\Gm \to G_{k}$ (see~\cite[\S 1.1]{kisin:abelian} for the terminology).

\subsection{}\label{subsec:M_0props}
With $M_0$ we can associate a Breuil window of level $0$ $(\mathfrak{P}_0,\Fil^1 \mathfrak{P}_0,\tilde{\varphi}_0)$ as follows: Let $V_{M_0}: M_0\to \sigma^*M_0$ be the Verschiebung, so that $\varphi_{M_0}\circ V_{M_0}$ is the multiplication by $p$ endomorphism of $M_0$. We take $\mathfrak{P}_0 = M_0$ and $\Fil^1 \mathfrak{P}_0 = (\sigma^{-1})^*M_0$, which we view as a submodule of $\mathfrak{P}_0$ via the map $(\sigma^{-1})^*V_{M_0}$. We have
\[
\tilde{\varphi}_0: \sigma^*\Fil^1\mathfrak{P}_0 = M_0 \xrightarrow{\simeq} \mathfrak{P}_0.
\]

The sequence
\[
M_0\otimes k \xrightarrow{V_{M_0}\otimes 1}\sigma^*M_0\otimes k\xrightarrow{\varphi_{M_0}\otimes 1}M_0\otimes k
\]
is exact. Moreover, $\Fil^1(M_0\otimes k)\subset M_0\otimes k$ is characterized by the property that $\sigma^*\Fil^1(M_0\otimes k)$ is the kernel of $\varphi_{M_0}\otimes 1$.

Therefore, we find that we could have also defined $\Fil^1 \mathfrak{P}_0$ to be the pre-image of $\Fil^1(M_0\otimes k)$ in $\mathfrak{P}_0 = M_0$.

\subsection{}
Suppose that we are given a finite extension $L/K_0$ contained in $K_0^\alg$, and a lift $\mathcal{G}$ of $\mathcal{G}_0$ to a $p$-divisible group over $\co_L$. Let $L_0\subset L$ be the maximal unramified subextension with residue field $k'$, and fix a uniformizer $\pi_L\in L$ with associated monic Eisenstein $\mathcal{E}_L(u)\in L_0[u]$, so that the theory of \S\ref{sec:lau_stuff} applies with $K,K_0,\mathcal{E}(u)$ replaced with $L,L_0,\mathcal{E}_L(u)$. Therefore, we can associate with $\mathcal{G}$ the object $\mathfrak{M}(\mathcal{G})\in \mathrm{BT}_{/\mathfrak{S}_{k'}}$, where
\[
\mathfrak{S}_{k'} = W(k')\pow{u}.
\]
Set $\mathfrak{P}(\mathcal{G}) = \varphi^*\mathfrak{M}(\mathcal{G})$. Then we have canonical isomorphisms
\[
\mathfrak{P}(\mathcal{G})/u \mathfrak{P}(\mathcal{G})\xrightarrow{\simeq}M_0\;;\; \mathfrak{P}(\mathcal{G})/\mathcal{E}_L(u)\mathfrak{P}(\mathcal{G})\xrightarrow{\simeq}\mathbb{D}(\mathcal{G})(\co_L).
\]

By a standard argument\footnote{This is Dwork's trick: Take any lift, and repeatedly apply $\varphi$ to make it $\varphi$-invariant in the limit. See~\cite[(1.5.5)]{kisin:abelian}.}, the $\varphi$-invariant tensors 
\[
\{1\otimes \bm{s}_{\alpha,0}\}\subset W(k')\otimes_WM_0^\otimes
\]
lift uniquely to $\varphi$-invariant tensors
\[
\{\bm{s}_{\alpha,\mathfrak{S},\mathcal{G}}\}\subset L_0\pow{u}\otimes_{\mathfrak{S}_{k'}}\mathfrak{P}(\mathcal{G})^\otimes.
\]
Here, we are viewing $\mathfrak{P}(\mathcal{G})$ as an object in the category of Breuil-Kisin modules over $\co_L$, and it is over this category that the ind-object $\mathfrak{P}(\mathcal{G})^\otimes$ is defined.

\begin{defn}
\label{defn:adapted}
We will say that $\mathcal{G}$ is \defnword{$G$-adapted} or \defnword{adapted to $G$} if the following conditions hold: 

\begin{enumerate}
 \item The tensors $\{\bm{s}_{\alpha,\mathfrak{S},\mathcal{G}}\}$ lie in $\mathfrak{P}(\mathcal{G})^\otimes$;
 \item There is an isomorphim
 \[
 \mathfrak{S}_{k'}\otimes_WM_0 \xrightarrow{\simeq}\mathfrak{P}(\mathcal{G})
 \]
 carrying, for each index $\alpha$, $1\otimes\bm{s}_{\alpha,0}$ to $\bm{s}_{\alpha,\mathfrak{S},\mathcal{G}}$, so that the stabilizer of $\{s_{\alpha,\mathfrak{S},\mathcal{G}}\}$ in $\GL(\mathfrak{P}(\mathcal{G}))$ can be identified with $G_{\mathfrak{S}_{k'}}$;
 \item The Hodge filtration on 
 \[
  \mathbb{D}(\mathcal{G})(\co_L)  = \co_L\otimes_{\varphi,\mathfrak{S}_{k'}}\mathfrak{P}(\mathcal{G})
 \]
 is split by a cocharacter $\mu_{\mathcal{G}}:\Gm\to G_{\co_L}$ lifting $\mu_0$.
 \end{enumerate} 
\end{defn}






We can now state the main technical result of this note.
\begin{proposition}
\label{prop:adapted}
Let $R^{\mathrm{univ}}$ be the universal deformation ring for $\mathcal{G}_0$ over $W$. Let $P_k\subset G_k$ be the stabilizer of the Hodge filtration
\[
\Fil^1(M_0\otimes k)\subset M_0\otimes k.
\]

Then there is a quotient $R^{\mathrm{univ}}_G$ of $R^{\mathrm{univ}}$ that is formally smooth over $W$ of dimension 
\[
d = \dim G_k-\dim P_k, 
\]
and is characterized by the following property: Given a finite extension $L/K_0$, a map $x:R^{\mathrm{univ}}\to \co_L$ factors through $R^{\mathrm{univ}}_G$ if and only if the corresponding lift $\mathcal{G}_x$ of $\mathcal{G}_0$ over $\co_L$ is $G$-adapted.
\end{proposition}

When $p>2$ or when $\mathcal{G}_0$ is connected, this is the content of~\cite[(1.5.8)]{kisin:abelian}. For the general case, we will essentially repeat the same line of reasoning, except that we will replace the use of Grothendieck-Messing style crystalline deformation theory with the results of Lau summarized in Theorem~\ref{thm:kisin_lau_p_divisible}.

\subsection{}
For the proof of Proposition~\ref{prop:adapted}, we will need a notion of $G$-adaptedness for lifts $\mathcal{G}_a$ of $\mathcal{G}_0$ over $\co_L/\pi_L^{a+1}$. This is defined just as for lifts over $\co_L$. Set $\mathfrak{S}_{k',a} = \mathfrak{S}_{k'}/u^{a+1}\mathfrak{S}_{k'}$. 

Associated with the lift is the object
\[
\mathfrak{P}_a(\mathcal{G}_a) = (\mathfrak{P}_a(\mathcal{G}_a),\Fil^1 \mathfrak{P}_a(\mathcal{G}_a),\tilde{\varphi}_a)
\]
in $\mathrm{BT}_{/\mathfrak{S}_{k',a}}$.

Now, note that $\tilde{\varphi}_a$ induces an isomorphism
\[
\varphi^*\mathfrak{P}_a(\mathcal{G}_a)[\mathcal{E}_L(u)^{-1}]\xrightarrow{\simeq}\mathfrak{P}_a(\mathcal{G}_a)[\mathcal{E}_L(u)^{-1}].
\]

There are now unique lifts
\[
\{\bm{s}_{\alpha,\mathfrak{S}_a,\mathcal{G}_a}\}\subset L_0\pow{u}/(u^{a+1})\otimes_{\mathfrak{S}_{k',a}}\mathfrak{P}_a(\mathcal{G}_a)^\otimes
\]
of $\{1\otimes\bm{s}_{\alpha,0}\}\subset W(k')\otimes_WM_0^\otimes$, which satisfy $\tilde{\varphi}_a(\varphi^*\bm{s}_{\alpha,\mathfrak{S}_a,\mathcal{G}_a}) = \bm{s}_{\alpha,\mathfrak{S}_a,\mathcal{G}_a}$. In other words, the tensors $\{\bm{s}_{\alpha,\mathfrak{S}_a,\mathcal{G}_a}\}$ are the unique $\varphi$-invariant lifts of $\{\bm{s}_{\alpha,0}\}$.

Set $\mathfrak{P}_{\dR}(\mathcal{G}_a) = \mathfrak{P}_a(\mathcal{G}_a)/\mathcal{E}_L(u)\mathfrak{P}_a(\mathcal{G}_a)$: This is a finite free module over $\co_L/\pi_L^{a+1}$. Equip it with the $\co_L/\pi_L^{a+1}$-linear direct summand
\[
\Fil^1\mathfrak{P}_{\dR}(\mathcal{G}_a) = \Fil^1 \mathfrak{P}_a(\mathcal{G}_a)/\mathcal{E}_L(u)\mathfrak{P}_a(\mathcal{G}_a)\subset \mathfrak{P}_{\dR}(\mathcal{G}_a).
\]
This gives a $2$-step descending filtration $\Fil^\bullet\mathfrak{P}_{\dR}(\mathcal{G}_a)$ on $\mathfrak{P}_{\dR}(\mathcal{G}_a)$ concentrated in degrees $0$ and $1$.

\begin{defn}
\label{defn:adapted_level_a}
We will say that $\mathcal{G}_a$ is \defnword{$G$-adapted} or \defnword{adapted to $G$} if the following conditions hold: 

\begin{enumerate}
 \item The tensors $\{\bm{s}_{\alpha,\mathfrak{S}_a,\mathcal{G}_a}\}$ lie in $\mathfrak{P}_a(\mathcal{G}_a)^\otimes$;
 \item There is an isomorphim
 \[
 \mathfrak{S}_{k',a}\otimes_WM_0 \xrightarrow{\simeq}\mathfrak{P}_a(\mathcal{G}_a)
 \]
 carrying, for each index $\alpha$, $1\otimes\bm{s}_{\alpha,0}$ to $\bm{s}_{\alpha,\mathfrak{S}_a,\mathcal{G}_a}$, so that the stabilizer of $\{s_{\alpha,\mathfrak{S}_a,\mathcal{G}_a}\}$ in $\GL(\mathfrak{P}_a(\mathcal{G}_a))$ can be identified with $G_{\mathfrak{S}_{k',a}}$.
 \item The filtration $\Fil^\bullet\mathfrak{P}_{\dR}(\mathcal{G}_a)$ is split by a cocharacter $\mu_{\mathcal{G}}:\Gm\to G_{\co_L/\pi_L^{a+1}}$ lifting $\mu_0$.
 \end{enumerate} 
\end{defn}

\begin{lemma}
\label{lem:point count}
Set 
\[
\mathcal{D}_G(\co_L/\pi_L^{a+1})  = \{x:R^{\mathrm{univ}}\to \co_L/\pi_L^{a+1}:\;\mathcal{G}_x\text{ is $G$-adapted}\}.
\] 
Then $\mathcal{D}_G(\co_L/\pi_L^{a+1})$ has at most $q^{ad}$ elements, where $q = \# k'$ is the size of $k'$ and $d = \dim G_k - \dim P_k$.
\end{lemma}
\begin{proof}
Fix any cocharacter $\mu:\Gm \to G$ lifting $\mu_0$, and let $\Fil^1M_0\subset M_0$ be the corresponding lift of the Hodge filtration. We have a direct sum decomposition
\[
M_0 = \Fil^1M_0 \oplus M'_0,
\]
where $M'_0\subset M_0$ is the subspace on which $\mu(\Gm)$ acts trivially.


Fix an $x:R^{\mathrm{univ}}\to \co_L/\pi_L^{a+1}$ in $\mathcal{D}_G(\co_L/\pi_L^{a+1})$. This corresponds to a $p$-divisible group $\mathcal{G}_x$ over $\co_L/\pi_L^{a+1}$ lifting $\mathcal{G}_0$, and is such that the tensors $\{\bm{s}_{\alpha,\mathfrak{S}_{a},\mathcal{G}_x}\}$ satisfy the conditions of~\eqref{defn:adapted_level_a}.

Therefore, we have an isomorphism
\[
\xi_{a}: \mathfrak{S}_{k',a}\otimes_WM_0\xrightarrow{\simeq}\mathfrak{P}_{a}(\mathcal{G}_x)
\]
carrying, for each $\alpha$, $1\otimes \bm{s}_{\alpha,0}$ to $\bm{s}_{\alpha,\mathfrak{S}_{a},\mathcal{G}_x}$. We will use this isomorphism to identify $\mathfrak{P}_{a}(\mathcal{G}_x)$ with $\mathfrak{S}_{k',a}\otimes_WM_0$.

Moreover, the Hodge filtration $\Fil^\bullet\mathfrak{P}_{\dR}(\mathcal{G}_x)$ on $\mathfrak{P}_{\dR}(\mathcal{G}_x)$ is split by a cocharacter of $G_{\co_L/\pi_L^{a+1}}$. Lift $\Fil^\bullet\mathfrak{P}_{\dR}(\mathcal{G}_x)$ to a filtration $\Fil^\bullet( \mathfrak{S}_{k',a}\otimes_WM_0 )$ split by a cocharacter of $G_{\mathfrak{S}_{k',a}}$. This is always possible by Proposition 1.1.5 of~\cite{kisin:abelian}. In addition, using Corollaire 3.3 of~\cite{Grothendieck1964-rn}, we find that, by replacing $\xi_{a}$ with $\xi_{a}\circ g$, for some $g\in G(\mathfrak{S}_{k',a})$ if necessary, we can assume that
\[
\Fil^\bullet ( \mathfrak{S}_{k',a}\otimes_WM_0 ) = \mathfrak{S}_{k',a}\otimes_W\Fil^\bullet M_0.
\]

We now have
\[
\xi_{a}^{-1}(\Fil^1 \mathfrak{P}_{a-1}(\mathcal{G}_x)) = (\mathcal{E}_L(u)\mathfrak{S}_{k',a}\otimes_WM'_0)\oplus (\mathfrak{S}_{k',a}\otimes_W\Fil^1M_0)\subset \mathfrak{S}_{k',a}\otimes_WM_0),
\]
and the isomorphism $\tilde{\varphi}_{a}:\varphi^*\Fil^1 \mathfrak{P}_{a}(\mathcal{G}_x)\xrightarrow{\simeq}\mathfrak{P}_{a}(\mathcal{G}_x)$ translates to an isomorphism
\[
\tilde{\varphi}_{a}:(\mathcal{E}_L(u)\mathfrak{S}_{k',a}\otimes_W\sigma^*M'_0)\oplus(\mathfrak{S}_{k',a}\otimes_W\sigma^*\Fil^1M_0)\xrightarrow{\simeq}\mathfrak{S}_{k',a}\otimes_WM_0,
\]
which lifts the isomorphism
\[
\tilde{\varphi}_0:p(\sigma^*M'_0)\oplus \sigma^*\Fil^1M_0\xrightarrow{\simeq}M_0
\]
obtained from the $F$-crystal structure on $M_0$ (see~\eqref{subsec:M_0props}) and is such that the induced isomorphism
\[
\varphi_a:(\mathfrak{S}_{k',a}\otimes_W\sigma^*M_0)[\mathcal{E}_L(u)^{-1}]\xrightarrow{\simeq}(\mathfrak{S}_{k',a}\otimes_WM_0)[\mathcal{E}_L(u)^{-1}]
\]
carries, for each $\alpha$, $1\otimes\sigma^*\bm{s}_{\alpha,0}$ to $1\otimes\bm{s}_{\alpha,0}$.

Let $\widetilde{\mathcal{D}}_G(\co_L/\pi_L^{a+1})$ be the set of such isomorphisms $\tilde{\varphi}_a$. Using Theorem~\ref{thm:kisin_lau_p_divisible}, we now find that we have constructed a surjection
\[
\widetilde{\mathcal{D}}_G(\co_L/\pi_L^{a+1})\to \mathcal{D}_G(\co_L/\pi_L^{a+1})
\]
carrying an isomorphism $\tilde{\varphi}_a$ to the $p$-divisible group over $\co_L/\pi_L^{a+1}$ corresponding to the Breuil window
\[
\Theta_a(\tilde{\varphi}_a)\coloneqq (\mathfrak{S}_{k',a}\otimes_WM_0,(\mathcal{E}_L(u)\mathfrak{S}_{k',a}\otimes_W M'_0)\oplus(\mathfrak{S}_{k',a}\otimes_W\Fil^1M_0),\tilde{\varphi}_a).
\]

Now, the action $\tilde{\varphi}_a\mapsto g\circ\tilde{\varphi}_a$, for $g\in G(\mathfrak{S}_{k',a})$ makes $\widetilde{\mathcal{D}}_G(\co_L/\pi_L^{a+1})$ a torsor under $G(\mathfrak{S}_{k',a})$. Fix $\tilde{\varphi}_a\in \mathcal{D}_G(\co_L/\pi_L^{a+1})$. Suppose now that we have $g_1,g_2\in \mathfrak{S}_{k',a}$ such that
\[
h \coloneqq g_2g_1^{-1}\in \mathcal{K}_a\coloneqq \ker\bigl(G(\mathfrak{S}_{k',a})\to G(\mathfrak{S}_{k',a-1})\bigr).
\]
Then $\varphi(h)\in G(\mathfrak{S}_{k',a})$. Therefore, it is easy to see that $h$ induces an isomorphism $\Theta_a(\tilde{\varphi}_a\circ g_1)\xrightarrow{\simeq}\Theta_a(\tilde{\varphi}_a\circ g_2)$ of Breuil windows if and only if it preserves the subspace
\[
(\mathcal{E}_L(u)\mathfrak{S}_{k',a}\otimes_W M'_0)\oplus(\mathfrak{S}_{k',a}\otimes_W\Fil^1M_0)\subset \mathfrak{S}_{k',a}\otimes_WM_0.
\]
Write $\mathcal{I}_a\subset \mathcal{K}_a$ for the subspace of elements that preserve this subspace.

Then we have shown that each non-empty fiber of the reduction map $\widetilde{\mathcal{D}}_G(\co_L/\pi_L^{a+1})\to \widetilde{\mathcal{D}}_G(\co_L/\pi_L^a)$ 
is in bijection with the set $\mathcal{K}_a/\mathcal{I}_a$. Therefore, to finish the proof of the lemma, it is enough to show that $\mathcal{K}_a/\mathcal{I}_a$ has $q^d$-elements. For this, simply observe that, if $P\subset G$ is the parabolic subgroup preserving the subspace $\Fil^1M_0\subset M_0$, then the reduction map $G(\mathfrak{S}_{k',a})\to G(\co_K/\pi_L^{a+1})$ identifies $\mathcal{K}_a/\mathcal{I}_a$ with
\begin{align*}
\frac{\ker(G(\co_L/\pi_L^{a+1})\to G(\co_L/\pi_L^a))}{\ker(P(\co_L/\pi_L^{a+1})\to P(\co_L/\pi_L^a))} \xrightarrow{\simeq}\frac{\pi^a\co_L}{\pi^{a+1}\co_L}\otimes_W\frac{\Lie G}{\Lie P} = \frac{\Lie G_k}{\Lie P_k}.
\end{align*}

\end{proof}

\begin{lemma}
\label{lem:one implication}
Suppose that $R_G$ is a quotient of $R^{\mathrm{univ}}$ that is formally smooth over $W$ of relative dimension $d$, and is such that, for all $x:R^{\mathrm{univ}}\to \co_L$ factoring through $R_G$, $\mathcal{G}_x$ is $G$-adapted. Then \underline{every} $x:R^{\mathrm{univ}}\to \co_L$ such that $\mathcal{G}_x$ is $G$-adapted factors through $R_G$.
\end{lemma}
\begin{proof}
Since $R_G$ is formally smooth, every map $x_a:R_G\to \co_L/\pi_L^{a+1}$ lifts to a map $x:R_G\to\co_L$. Therefore, the hypothesis implies that, for every $a\in \Z_{\geq 0}$, we have:
\begin{equation}\label{eqn:inclusion RG DG}
\{x_a:\;R^{\mathrm{univ}}\to \co_L/\pi_L^{a+1}:\;x\text{ factors through $R_G$}\}\subset \mathcal{D}_G(\co_L/\pi_L^{a+1}).
\end{equation}
The set on the left hand side has size $q^{ad}$, since $R_G$ is formally smooth of relative dimension $d$ over $W$. The set on the right has at most $q^{ad}$ elements by Lemma~\ref{lem:point count}. Therefore, the inclusion~\eqref{eqn:inclusion RG DG} has to be an equality for every $a$. This easily implies the corollary.
\end{proof}

\subsection{}
To finish the proof of Proposition~\ref{prop:adapted}, it is now enough to construct a formally smooth quotient $R_G$ of $R^{\mathrm{univ}}$ as in Lemma~\ref{lem:one implication}. For this, we follow a construction of Faltings~\cite{faltings:very_ramified}. Fix any cocharacter $\mu:\Gm \to G$ lifting $\mu_0$, and let $\Fil^1M_0\subset M_0$ be the corresponding lift of the Hodge filtration. We have a direct sum decomposition
\[
M_0 = \Fil^1M_0 \oplus M'_0,
\]
where $M'_0\subset M_0$ is the subspace on which $\mu(\Gm)$ acts trivially.

Let $P\subset G$ be the parabolic subgroup stabilizing $\Fil^1M_0$. Let $U^{\op}_G\subset G$ and $U^{\op}\subset\GL(M_0)$ be the \emph{opposite} unipotents determined by $\mu$, so that 
\[
\Lie U^{\mathrm{op}} \subset \End(M_0) , \quad \Lie U^{\mathrm{op}}_G\subset \Lie G_W
\]
are the $-1$-eigenspaces for $\mu$.

We have $\dim U^{\op}_G = d  = \dim G - \dim P$.

\subsection{}
Let $R$ (resp. $R_G$) be the complete local ring of $U^{\mathrm{op}}$ (resp. $U^{\mathrm{op}}_G$) at the identity.

Choose a basis for $\Lie U^{\mathrm{op}}_G$ and extend it to a basis for $\Lie U^{\mathrm{op}}$. This gives us compatible co-ordinates
\[
R_G\xrightarrow{\simeq}W\pow{u_1,\ldots,u_d} , \quad  R\xrightarrow{\simeq} W\pow{u_1,\ldots, u_d,u_{d+1},\ldots, u_n}.
\]
Let $\varphi:R\to R$ be the Frobenius lift satisfying $\varphi(u_i) = u_i^p$; it preserves $R_G$.

Now set $M^\diamond = R\otimes_WM_0$ (resp. $M^\diamond_{R_G} = R_G\otimes_WM_0$). Equip $M^\diamond$ and $M^\diamond_{R_G}$ with the constant Hodge filtrations
\[
\Fil^\bullet M^\diamond = R\otimes_W\Fil^\bullet M_0 , \quad  \Fil^\bullet M^\diamond_{R_G} = R_G\otimes_W\Fil^\bullet M_0.
\]
We will also equip $M^\diamond$ with the map
\[
\varphi_{M^\diamond} = g(1\otimes \varphi_0):\varphi^*M^\diamond \to M^\diamond,
\]
Here, $g\in U^{\mathrm{op}}(R)$ is the tautological element, and $\varphi_0:\varphi^*M_0\to M_0$ is the map obtained from the $F$-crystal structure on $M_0$. Note that $\varphi_{M^{\diamond}}$ induces a map
\[
\varphi_{M^\diamond_{R_G}}: \varphi^*M^\diamond_{R_G}\to M^\diamond_{R_G}.
\]

\begin{proposition}
\label{prop:dieudonne faltings construction}
There exists a choice of the lift $\mu$, for which there is a $p$-divisible group $\mathcal{G}^\diamond$ over $R$ deforming $\mathcal{G}_0$ and equipped with an isomorphism:
\[
\mathbb{D}(\mathcal{G}^\diamond)(R) \xrightarrow{\simeq} M^\diamond
\]
of filtered $R$-modules equipped with $\varphi$-semilinear endomorphisms. The associated map $R^{\mathrm{univ}}\to R$ of local $W$-algebras obtained from the universal property of $R^{\mathrm{univ}}$ is an isomorphism.
\end{proposition}
\begin{proof}
First, choose any $G$-adapted lift $\mathcal{G}_W$ of $\mathcal{G}_0$ over $W$. Such a lift always exists, as shown in the proof of~\cite[Proposition 1.1.13]{kisin:lr}.

We then have a canonical isomorphism
\[
\mathbb{D}(\mathcal{G}_W)(W)\xrightarrow{\simeq}M_0,
\]
and the image of the Hodge filtration $\Fil^1\mathbb{D}(\mathcal{G}_W)(W)$ in $M_0$ will give a filtration $\Fil^1M_0$ that will be split by a cocharacter $\mu:\Gm\to G$ lifting $\mu_0$.

The deformation $\mathcal{G}_W$ corresponds to a map $R^{\mathrm{univ}}\to W$. Let $\mathcal{G}^{\mathrm{univ}}$ be the universal deformation of $\mathcal{G}_0$ over $R^{\mathrm{univ}}$. Then there is an isomorphism
\begin{equation}\label{eqn:dieudonne reduction}
W\otimes_{R^{\mathrm{univ}}}\mathbb{D}(\mathcal{G}^{\mathrm{univ}})(R^{\mathrm{univ}})\xrightarrow{\simeq}M_0
\end{equation}
of filtered $F$-crystals over $W$.

Now, Theorem 10 of~\cite{faltings:very_ramified}, which is a purely linear algebraic result, implies that there is a unique map $f: R^{\mathrm{univ}}\to R$ such that if $e:R\to W$ corresponds to the identity section of $U^{\mathrm{op}}$, then the composition 
\[
R^{\mathrm{univ}}\to R\to W
\]
corresponds to lift $\mathcal{G}_W$ of $\mathcal{G}_0$, and such that there exists an isomorphism
\[
R\otimes_{R^{\mathrm{univ}}}\mathbb{D}(\mathcal{G}^{\mathrm{univ}})(R^{\mathrm{univ}})\xrightarrow{\simeq}M^\diamond
\]
of filtered $R$-modules lifting~\eqref{eqn:dieudonne reduction} and compatible with the $F$-crystal structure in a sense that is made precise in \emph{loc.~cit.}

That this map is an isomorphism follows from a versality argument; see~\cite[\S1.4.2]{mp:thesis} for details.
\end{proof}

By the above proposition, we can identify $R_G$ with a quotient of $R^{\mathrm{univ}}$. The proof of Proposition~\ref{prop:adapted} is now completed by
\begin{proposition}
\label{prop:RG adapted}
For every $x:R\to \co_L$ factoring through $R_G$, $\mathcal{G}_x$ is $G$-adapted.
\end{proposition}
\begin{proof}
This is essentially shown in Proposition 1.1.13 of~\cite{kisin:lr}. We recall the reasoning here.

By construction, for each index $\alpha$, the constant tensor 
\[
\bm{s}_{\alpha,R_G} = 1\otimes\bm{s}_{\alpha,0}\in (M^\diamond_{R_G})^\otimes = \mathbb{D}(\mathcal{G}^\diamond)(R_G)^\otimes
\]
is $\varphi$-invariant, and lies in $\Fil^0(M^\diamond_{R_G})^\otimes$. Moreover, the Hodge filtration 
\[
\Fil^1\mathbb{D}(\mathcal{G}^\diamond)(R_G)\subset \mathbb{D}(\mathcal{G}^\diamond)(R_G)
\]
is split by the cocharacter $1\otimes\mu:\Gm\to R_G\otimes_WG = G_{R_G}$.

For every point $x:R_G\to \co_L$, the specialization of $\bm{s}_{\alpha,R_G}$ gives us a tensor $\bm{s}_{\alpha,\dR,x}\in \mathbb{D}(\mathcal{G}_x)(\co_L)^\otimes$. The stabilizer of $\{\bm{s}_{\alpha,\dR,x}\}$ is isomorphic to $G_{\co_L}$, and the Hodge filtration on $\mathbb{D}(\mathcal{G}_x)(\co_L)$ is split by a cocharacter $\mu_x:\Gm\to G_{\co_L}$ lifting $\mu_0$.

Let $T_p(\mathcal{G}^{\diamond})$ be the lisse $p$-adic sheaf over $\Spec R_G[p^{-1}]$ obtained from the $p$-adic Tate module of $\mathcal{G}^{\diamond}$. By Lemma 1.1.17 of~\cite{kisin:lr}, we can find a global section
\[
\bm{s}_{\alpha,p,R_G}\in H^0\bigl(\Spec R_G[p^{-1}],T_p(\mathcal{G}^{\diamond})^\otimes[p^{-1}]\bigr)
\]
whose specialization at any point $x:R_G\to \co_L$ gives a $\Gamma_L$-invariant tensor $\bm{s}_{\alpha,p,x}\in T_p(\mathcal{G}_x)^\otimes[p^{-1}]$, characterized by the property that its de Rham realization is $\bm{s}_{\alpha,\dR,x}$, and its crystalline realization is $\bm{s}_{\alpha,0}$. 

We claim that $\bm{s}_{\alpha,p,R_G}$ is a section of $T_p(\mathcal{G}^{\diamond})^\otimes$. To see this, we can check at any point of $\Spec R_G[p^{-1}]$, which we choose to be the identity $x_0\in \widehat{U}_G(W)$. Then, by the construction in Proposition~\ref{prop:dieudonne faltings construction}, $\mathcal{G}_{x_0}$ is $G$-adapted. Therefore, for each $\alpha$, we have a $\varphi$-invariant tensor $\bm{s}_{\alpha,\mathfrak{S},x_0}\in \mathfrak{P}(\mathcal{G}_{x_0})^\otimes$, which, via the functor of~\eqref{subsec:kisin_functor}, corresponds to the tensor $\bm{s}_{\alpha,p,x_0}$. In particular, this implies that $\bm{s}_{\alpha,p,x_0}$ belongs to $T_p(\mathcal{G}_{x_0})^\otimes$, as desired. Moreover, by Remark~\ref{rem:key lemma other way} and Theorem~\ref{thm:kisin_lau_p_divisible}, there exists an isomorphism
\[
\mathfrak{S}\otimes_{\Z_p}T_p(\mathcal{G}_{x_0})^\vee \xrightarrow{\simeq}\mathfrak{P}(\mathcal{G}_{x_0})
\]
carrying $1\otimes \bm{s}_{\alpha,p,x_0}$ to $\bm{s}_{\alpha,\mathfrak{S},x_0}$, for all $\alpha$. 

Let $G_{\Z_p}\subset \GL(T_p(\mathcal{G}_{x_0}))$ be the stabilizer of $\{\bm{s}_{\alpha,p,x_0}\}$: This is a reductive group over $\Z_p$, isomorphic over $W$ to $G$. It now follows that, for \emph{every} $x:R_G\to \co_L$, the $\Gamma_L$-invariant tensor $\bm{s}_{\alpha,p,x}$ lies in $T_p(\mathcal{G}_x)^\otimes$. Moreover, the stabilizer of the collection $\{\bm{s}_{\alpha,p,x}\}$ in $\GL(T_p(\mathcal{G}_x))$ is isomorphic to $G_{\Z_p}$. Therefore, it now follows from Theorems~\ref{thm:kisin_key_lemma} and~\ref{thm:kisin_lau_p_divisible} that $\mathcal{G}_x$ is adapted to $G$.
\end{proof}

\section{Integral canonical models}

\subsection{}
Let $(G,X)$ be a Shimura datum, and let $H$ be a faithful representation of $G$ over $\Q$, equipped with a symplectic pairing $H\times H\to\Q$, affording an embedding of reductive $\Q$-groups
\[
G\hookrightarrow \GSp(H,\psi).
\]
into the groups of symplectic similitudes for $(H,\psi)$. We assume that, for each $x\in X$, the associated homomorphism $h_x:\mathbb{S}\to G_\R$ induces a Hodge structure on $H_\C$ of weights $(-1,0),(0,-1)$, which is polarized by $\psi$.

Fix a $\Z$-lattice $H_\Z\subset H$ and let $K\subset G(\A_f)$ be a neat compact open such that $K$ stabilizes $H_{\widehat{\Z}} = H_\Z\otimes\widehat{\Z}\subset H_{\A_f}$. Associated with this is the Shimura variety $\Sh_K = \Sh_K(G,X)$ over the reflex field $E = E(G,X)$. 

Let 
\[
H^\vee_\Z = \{h\in H_\Q:\;\psi(h,H_\Z)\subset \Z\}
\]
be the dual lattice for $H_\Z$ with respect to $\psi$, and let $m\in\Z_{\geq 1}$ be such that $m^2=[H^\vee_\Z:H_\Z]$ is the discriminant of $H_\Z$.

\subsection{}
A classical construction (see for instance~\cite[(2.1.8)]{mp:toroidal}) associates with $(H,\psi)$ and the lattice $H_\Z\subset H$, a polarized variation of (pure) $\Z$-Hodge structures over $\Sh_K(\C)$ of weights $(0,-1),(-1,0)$, and thus a family of polarized abelian varieties $A_{\Sh_K(\C)}\to \Sh_K(\C)$. The theory of canonical models for Shimura varieties now implies that this family arises from a canonical polarized abelian scheme $A\to \Sh_K$, which corresponds to a finite and unramified map
\begin{equation}\label{eqn:siegel_embedding}
\Sh_K \to \mathcal{X}_{d,m,\Q},
\end{equation}
where, $\mathcal{X}_{d,m}$ is the moduli stack over $\Z$ of polarized abelian schemes of dimension $d = \frac{1}{2}\dim H$ and degree $m^2$.

\subsection{}\label{subsec:derham realization}
More generally, the classical construction referenced above associates with a pair $(N,N_{\widehat{\Z}})$ consisting of an algebraic $\Q$-representation $N$ of $G$ and a $K$-stable lattice $N_{\widehat{\Z}}\subset N_{\A_f}$ a variation of pure $\Z$-Hodge structures
\[
(\bm{N}_B,\bm{N}_{\dR,\Sh_K(\C)},\Fil^\bullet\bm{N}_{\dR,\Sh_K(\C)}),
\]
where $\bm{N}_B$ is the underlying $\Z$-local system over $\Sh_K(\C)$, $\bm{N}_{\dR,\Sh_K(\C)} = \co_{\Sh_K(\C)}\otimes\bm{N}_B$ is the associated vector bundle with integrable connection, and $\Fil^\bullet\bm{N}_{\dR,\Sh_K(\C)}$ is the filtration on it by sub vector bundles inducing a pure Hodge structure at every point of $\Sh_K(\C)$. 

The theory of canonical models implies that, for every prime $\ell$, the associated $\ell$-adic local system $\Z_\ell\otimes\bm{N}_B$ descends to a lisse $\ell$-adic sheaf $\bm{N}_{\ell}$ over $\Sh_K$. When $(N,N_{\widehat{\Z}}) = (H,H_{\widehat{\Z}})$, this is simply the $\ell$-adic Tate module associated with $A$. 

More interestingly, Deligne's theory of absolute Hodge structures (see~\cite[\S 2.2]{kisin:abelian}) shows that the filtered vector bundle with integrable connection $(\bm{N}_{\dR,\Sh_K(\C)},\Fil^\bullet\bm{N}_{\dR,\Sh_K(\C)})$ has a canonical descent
\[
(\bm{N}_{\dR,\Sh_K},\Fil^\bullet\bm{N}_{\dR,\Sh_K})
\]
to a filtered vector bundle over $\Sh_K$ with integrable connection.

\subsection{}
Fix a prime $p$, and a place $v\vert p$ of $E$. Given an algebraic stack $\mathcal{X}$ over $\co_{E,(v)}$, and a normal algebraic stack $Y$ over $E$ equipped with a finite map $j_{E}:Y\to\mathcal{X}_{E}$, the \emph{normalization} of $\mathcal{X}$ in $Y$ is the finite $\mathcal{X}$-stack $j:\mathcal{Y}\to\mathcal{X}$, characterized by the property that $j_*\co_{\mathcal{Y}}$ is the integral closure of $\co_{\mathcal{X}}$ in $(j_{E})_*\co_Y$. It is also characterized by the following universal property: given a finite morphism $\mathcal{Z}\to\mathcal{X}$ with $\mathcal{Z}$ a normal algebraic stack, flat over $\co_{E,(v)}$, any map of $\mathcal{X}_{E}$-stacks $\mathcal{Z}_{E}\to Y$ extends uniquely to a map of $\mathcal{X}$-stacks $\mathcal{Z}\to\mathcal{Y}$.

\subsection{}\label{subsec:SK const}
We now obtain an integral model $\mathcal{S}_K$ for $\Sh_K$ over $\co_{E,(v)}$ by taking the normalization of $\co_{E,(v)}\otimes_{\Z}\mathcal{X}_{d,m}$ in $\Sh_K$. By construction $A$ extends to a polarized abelian scheme over $\mathcal{S}_K$, which we denote once again by $A$. 

Fix a finite extension $k$ of $k(v)$, and set $W=W(k)$, $L_0 = \mathrm{Frac}(W)$. Fix an algebraic closure $L_0^\alg$ of $K_0$. Suppose that we have a point $t\in \mathcal{S}_K(k)$. Let $\mathcal{G}_t = A_t[p^\infty]$ be the $p$-divisible group over $k$ associated with the fiber at $t$ of $A$. Let $\co_t$ be the complete local ring of $\mathcal{S}_K$, and let $R_t$ be the universal deformation ring of $\mathcal{G}_t$. Then $R_t$ is non-canonically isomorphic to a power series ring over $W$ in $g^2$-variables. By Serre-Tate deformation theory, $\co_t$ is the normalization of a quotient of $R_t$. 

\begin{proposition}
\label{prop:local model}
Suppose that $G$ admits a reductive model $G_{(p)}$ over $\Z_{(p)}$ such that $K_p = G_{(p)}(\Z_p)$. Then $\co_t$ is formally smooth over $W$.
\end{proposition}
\begin{proof}
Set $H_{(p)} = H_\Z\otimes\Z_{(p)}$, and fix any collection of tensors $\{s_{\alpha}\}\subset H_{(p)}^\otimes$ whose pointwise stabilizer is $G_{(p)}\subset\GL(H_{(p)})$.\footnote{It is always possible to find such a collection; see~\cite[Lemma 2.3.1]{kisin:abelian}, and also the footnote on p. 72 of~\cite{mp:toroidal}. More precisely, the first cited result shows the existence of the tensors under the assumption that $G$ has no factors of type $B$ when $p=2$, and the footnote explains why this assumption is unnecessary.} These tensors give rise to global sections $\{\bm{s}_{\alpha,p}\}\subset H^0(\Sh_K,\bm{H}_p^\otimes)$.

Choose any lift $\tilde{t}\in\mathcal{S}_K(\co_L)$ of $t$ to a point valued in the ring of integers $\co_L$ of a finite extension $L/L_0$ contained in $L_0^\alg$. By enlarging $k$ if necessary, we can assume that $L$ is totally ramified over $L_0$. Let $\Gamma_L = \Gal(L^\alg/L)$ be the absolute Galois group of $L$. We obtain $\Gamma_L$-invariant tensors $\{\bm{s}_{\alpha,p,\tilde{t}}\}\subset \bm{H}_{p,\tilde{t}}^\otimes$, whose stabilizer is isomorphic to $G_{(p),\Z_p}$. 

Set $\bm{H}_{\cris,t} = \mathbb{D}(\mathcal{G}_t)(W)$ and $\bm{H}_{\dR,\tilde{t}} = H^1_{\dR}(A_{\tilde{t}}/\co_L)^\vee$. Then Theorems~\ref{thm:kisin_key_lemma} and~\ref{thm:kisin_lau_p_divisible} give us $\varphi$-invariant tensors
\[
\{\bm{s}_{\alpha,\cris,t}\}\subset \bm{H}_{\cris,t}^\otimes,
\]
whose stabilizer is isomorphic to $G_W$, as well as tensors
\[
\{\bm{s}_{\alpha,\dR,\tilde{t}}\}\subset \Fil^0\bm{H}_{\dR,\tilde{t}}^\otimes,
\]
whose stabilizer is isomorphic to $G_{\co_L}$. Moreover, the Hodge filtration on $\bm{H}_{\dR,\tilde{t}}$ can be split via a cocharacter of $G_{\co_L}$.

By a theorem of Blasius-Wintenberger~\cite{blasius:padic} on the compatibility of the de Rham comparison isomorphism with homological realizations of Hodge cycles on abelian varieties, the tensors $\{\bm{s}_{\alpha,\dR,\tilde{t}}\}$ coincide with those obtained from $\{s_{\alpha}\}$ via the de Rham functor described in~\eqref{subsec:derham realization}.

Also, by construction, under the canonical identification
\[
\bm{H}_{\cris,t}\otimes_Wk = H^1_{\dR}(A_t/k)^\vee = \bm{H}_{\dR,\tilde{t}}\otimes_{\co_L}k,
\] 
the tensors $\{\bm{s}_{\alpha,\cris,t}\otimes 1\}$ are carried to $\{\bm{s}_{\alpha,\dR,\tilde{t}}\otimes 1\}$. In particular, the Hodge filtration on 
\[
\bm{H}_{\dR,t} \coloneqq H^1_{\dR}(A_t/\F_p^{\alg})^\vee 
\]
is split by a cocharacter $\mu_0:\mathbb{G}_m\to G_{\F_p^{\alg}}$. 

This implies that we can apply the theory of Section~\ref{sec:deformation} with $\mathcal{G}_0=\mathcal{G}_t$, and $\{\bm{s}_{\alpha,0}\} = \{\bm{s}_{\alpha,\cris,t}\}$. Therefore, by Proposition~\ref{prop:adapted}, we have a canonical formally smooth quotient $R_G^{\mathrm{univ}}$ of $R_t$ characterized by the property that any point $\tilde{t}':R_t\to \co_L$ factors through $R_G^{\mathrm{univ}}$ if and only if the corresponding lift $\mathcal{G}_{\tilde{t}'}$ is $G_W$-adapted. We claim that $R_G^{\mathrm{univ}}$ is identifed with $\co_t$. 

This is done precisely as in the proof of Proposition 2.3.5 of~\cite{kisin:abelian}. First, an easy dimension count, using the fact that $\mu_0$ is conjugate to the inverse of the Hodge cocharacter associated with the Shimura datum $(G,X)$, shows that it is enough to prove that $\co_t$ is a quotient of $R_G^{\mathrm{univ}}$. In other words, we must show that every map $\tilde{t}':\co_t\to \co_L$ is $G$-adapted. But, every such $\tilde{t}'$ gives us tensors $\{\bm{s}_{\alpha,p,\tilde{t}'}\}\subset \bm{H}_{p,\tilde{t}'}^\otimes$, which give rise, as in Theorem~\ref{thm:kisin_key_lemma}, to $\varphi$-invariant tensors $\{\bm{s}_{\alpha,\mathfrak{S},\tilde{t}'}\}\subset \mathfrak{M}(\mathcal{G}_{\tilde{t}'})^\otimes$, whose stabilizer is isomorphic to $G_{\mathfrak{S}}$. 

These tensors in turn give rise in turn to $\varphi$-invariant tensors $\{\bm{s}_{\alpha,\cris,\tilde{t}'}\}\subset \bm{H}_{\cris,t}^\otimes$, and de Rham tensors $\{\bm{s}_{\alpha,\dR,\tilde{t}'}\}\subset \Fil^0\bm{H}_{\dR,\tilde{t}'}^\otimes$.

To finish the proof, it only remains to check that the collection $\{\bm{s}_{\alpha,\cris,\tilde{t}}\}$ coincides with the collection $\{\bm{s}_{\alpha,\cris,t}\}$ constructed from the initial choice of lift $\tilde{t}$. This follows as in the proof of~\cite[Proposition 2.3.5]{kisin:abelian}, using a parallel transport argument.
\end{proof}

The explicit construction of $R_G^{\mathrm{univ}}$ in Section~\ref{sec:deformation} now implies (see, for instance, Corollary 4.13 of~\cite{mp:reg})
\begin{proposition}\label{prop:de_rham_realization}
The functor from algebraic $\Q$-representations $N$ of $G$ to filtered vector bundles $(\bm{N}_{\dR,\Sh_K},\Fil^\bullet \bm{N}_{\dR,\Sh_K})$ extends canonically to an exact tensor functor $N\to\bm{N}_{\dR}$ from algebraic $\Z_{(p)}$-representations $N$ of $G_{(p)}$ to filtered vector bundles on $\mathcal{S}_{K}$ equipped with an integrable connection. When $N = H_{(p)}$, the associated filtered vector bundle with integrable connection is simply $\bm{H}_{\dR}$, the relative first de Rham homology of $A \to \mathcal{S}_{K}$.
\end{proposition}

\subsection{}
Write $\widehat{\mathcal{S}}_{K}$ for the formal completion of $\mathcal{S}_{K}$ along $\mathcal{S}_{K,k(v)}$. The relative first crystalline cohomology of $A$ over $\mathcal{S}_{K,(v)}$ gives a Dieudonn\'e $F$-crystal $\bm{H}^\vee_{\cris}$ over $\mathcal{S}_{K,k(v)}$ whose evaluation on $\widehat{\mathcal{S}}_K$ is canonically isomorphic to the $p$-adic completion of $\bm{H}^\vee_{\dR}$ as a vector bundle with integrable connection.

We will now expand our definition of an $F$-crystal over $\mathcal{S}_{K,k(v)}$ to mean a crystal of vector bundles $\bm{N}$ over $\mathcal{S}_{K,k(v)}$ equipped with an isomorphism $\mathrm{Fr}^*\bm{N} \xrightarrow{\simeq}\bm{N}$ in the $\Q_p$-linear \emph{isogeny} category associated with the category of crystals over $\mathcal{S}_{K,k(v)}$.

Given an algebraic $\Z_{(p)}$-representation $N_{(p)}$ of $G_{(p)}$, the restriction of the vector bundle $\bm{N}_{\dR}$ to $\widehat{\mathcal{S}}_K$ is the evaluation of a canonical crystal $\bm{N}_{\cris}$ of vector bundles over $\mathcal{S}_{K,k(v)}$. Since $H_{(p)}$ is a faithful representation of $G_{(p)}$, we now have

\begin{proposition}
\label{prop:crystalline_realization}
There is a unique structure of an $F$-crystal on $\bm{N}_{\cris}$, which when $N_{(p)}= H^\vee_{(p)}$ agrees with the canonical $F$-crystal structure on the Dieudonn\'e $F$-crystal $\bm{H}^\vee_{\cris}$. This gives an exact tensor functor $N_{(p)}\mapsto \bm{N}_{\cris}$ from $\Z_{(p)}$-representations of $G_{(p)}$ to $F$-crystals over $\mathcal{S}_{K,k(v)}$.
\end{proposition}

As in~\cite[\S~3]{kisin:abelian} (see especially Corollary 3.4.14), Proposition~\ref{prop:local model} implies
\begin{theorem}
\label{theorem:canonical models}
Let $(G,X)$ be a Shimura datum of abelian type, and let $K_p\subset G(\Q_p)$ be a hyperspecial compact open subgroup. Then, for any place $v\vert p$ of the reflex field $E = E(G,X)$, the pro-Shimura variety
\[
\Sh_{K_p} = \varprojlim_{K^p\subset G(\A_f^p)}\Sh_{K_pK^p}
\]
over $E$ admits an integral canonical model $\mathcal{S}_{K_p}$ over $\co_{E,(v)}$. That is, $\mathcal{S}_{K_p}$ is regular and formally smooth over $\co_{E,(v)}$ with generic fiber $\Sh_{K_p}$, and, given any other regular and formally smooth scheme $S$ over $\co_{E,(v)}$, any map $S_E\to\Sh_{K_p}$ extends to a map $S\to\mathcal{S}_{K_p}$.

When $(G,X)$ is of Hodge type, so that $G$ admits a polarizable faithful representation $H$ of weights $(0,-1),(-1,0)$, we have
\[
\mathcal{S}_{K_p} = \varprojlim_{K^p\subset G(\A_f^p)}\mathcal{S}_{K_pK^p}
\]
where $\mathcal{S}_{K_pK^p}$ is constructed using the representation $H$ as in~\eqref{subsec:SK const}.
\end{theorem}

\section{The Tate conjecture in characteristic $2$}\label{sec:tate}

In this section, we show how to remove the hypothesis $p\neq 2$ from the results of~\cite{mp:tatek3}. We will assume that the reader is familiar with the methods of that paper, and will only indicate the necessary changes. The material in this section is due to the second named author.

\subsection{}
We will need a particular class of Shimura varieties of abelian type associated with a quadratic lattice $L$ of signature $(n,2)$ with $n\geq 1$. This theory is summarized in~\cite[\S 4]{mp:tatek3}. 

From $L$ we obtain a Shimura datum $(G_L,X_L)$, where $G_L = \SO(L_\Q)$ and $X_L$ is the symmetric domain of oriented negative definite planes in $L_\R$. Associated with $L$ is also the discriminant kernel $K_L\subset G_L(\A_f)$, which is the largest subgroup that stabilizes $L_{\widehat{\Z}}$ and acts trivially on the discriminant module $L^\vee/L$. 

Associated with any compact open subgroup $K\subset G(\A_f)$ is the Shimura variety (or rather algebraic stack)
\[
\Sh_K(L) \coloneqq \Sh_{K}(G_L,X_L),
\]
which is defined over $\Q$. When $K=K_L$, we will set $\Sh(L) = \Sh_{K_L}(L)$.

As in \emph{loc. cit.}, we will assume that $L$ contains a hyperbolic plane, so that we have
\[
\Sh(L)(\C) = \Gamma_L\backslash X_L,
\]
where $\Gamma_L\subset \SO(L)(\Z)$ is the discriminant kernel. 

\subsection{}\label{subsec:motives char 0}
The Shimura variety $\Sh(L)$ carries a canonical family of $\Z$-motives\footnote{Here, this means absolute Hodge motives; see~\cite[\S 2]{mp:tatek3}} $\bm{L}$ associated with the standard representation $L_\Q$ of $G_L$, and the lattice $L_{\widehat{\Z}}$. More precisely, for every point $s\to \Sh(L)$, the cohomological realizations of the $\Z$-motive $\bm{L}_s$ will be the fibers at $s$ of the sheaves $\bm{L}_B$, $\bm{L}_{\ell}$, and $\bm{L}_{\dR}$ associated with $(L_\Q,L_{\widehat{\Z}})$.

Moreover, as in~\cite[\S~4.4]{mp:tatek3}, we find that that there exists a finite \'etale cover $\widetilde{\Sh}(L)\to \Sh(L)$ associated with the central extension $\GSpin(L_\Q)\to G_L$, and an abelian scheme $A^\mathrm{KS}\to \widetilde{\Sh}(L)$, whose realizations are associated with the left regular representation $H$ of $\GSpin(L_\Q)$ on the Clifford algebra $C(L_\Q)$, and thus give us the family of motives $\bm{H}$ over $\widetilde{\Sh}(L)$. The family of motives $\underline{\End}(\bm{H})$ associated with $\End(H)$ descends over $\Sh(L)$, as does the sheaf $\bm{E}$ of endomorphisms of the abelian scheme $A^\mathrm{KS}$, along with the homological realization map
\[
\bm{E} \to \underline{\End}(\bm{H}).
\]

The action of $L$ on the Clifford algebra $C(L)$ by left multiplication induces an embedding of families of motives
\[
\bm{L} \hookrightarrow \underline{\End}(\bm{H})
\]
over $\Sh(L)$.

Given any scheme $T\to \Sh(L)$, a \defnword{special endomorphism} over $T$ will be a section of $\bm{E}$ over $T$, whose homological realizations land in the image of $\bm{L}$ at all geometric points of $T$. We will write $L(T)$ for the group of special endomorphisms over $T$. Composition in $\underline{\End}(\bm{H})$ induces a canonical positive definite quadratic form on $L(T)$ with values in $\Z$.

\subsection{}\label{subsec:diamond}
Let $L^\diamond$ be another quadratic lattice of signature $(n^\diamond,2)$ such that we have an isometric embedding $L\hookrightarrow L^\diamond$ onto a direct summand of $L^\diamond$. Let $\Lambda\subset L^\diamond$ be the orthogonal complement of $L$, so that we can view $G_L$ as the subgroup of $G_{L^\diamond}$ that acts trivially on $\Lambda$. We then have an embedding of Shimura data
\[
(G_L,X_L)\hookrightarrow (G_{L^\diamond},X_{L^\diamond})
\]
and thus a map $\Sh(L)\to \Sh(L^\diamond)$ of Shimura varieties. 

For any $T\to\Sh(L^\diamond)$, write $L^\diamond(T)$ for the associated space of special endomorphisms over $T$. Then there is a canonical isometric embedding $\Lambda\hookrightarrow L^\diamond(\Sh(L))$ such that, for any $T\to \Sh(L)$, we have a canonical identification
\[
L(T) = \Lambda^\perp\subset L^\diamond(T).
\]

\subsection{}\label{subsec:self-dual}
Suppose now that $L_{(p)} \coloneqq L_{\Z_{(p)}}$ is self-dual. Then the associated special orthogonal group $G_{(p)} \coloneqq \SO(L_{(p)})$ is a reductive model for $G_L$ over $\Z_{(p)}$, and $K_{L,p} = G_{(p)}(\Z_p)$ is a hyperspecial compact open subgroup of $G_L(\Q_p)$. Therefore, by Theorem~\ref{theorem:canonical models}, we have a smooth integral canonical model $\mathcal{S}(L)_{(p)}$ over $\Z_{(p)}$ for $\Sh(L)$. More precisely, the theorem gives us an integral canonical model $\mathcal{S}_{K^pK_{L,p}}(L)_{(p)}$ over $\Z_{(p)}$ for $\Sh_{K^pK_{L,p}}(L)$, when $K^p\subset G_L(\A_f^p)$ is a sufficiently small neat compact open subgroup. We obtain $\mathcal{S}(L)_{(p)}$ as the quotient stack of this integral model by the natural action of the finite group $K^p_L/K^p$. 

By Proposition~\ref{prop:crystalline_realization}, we now have crystalline realizations $\bm{L}_{\cris}$ and $\underline{\End}(\bm{H}_{\cris})$ of $\bm{L}$ and $\underline{\End}(\bm{H})$, respectively: These are $F$-crystals over the special fiber $\mathcal{S}(L)_{\F_p}$ of $\mathcal{S}(L)_{(p)}$, and we have a canonical inclusion $\bm{L}_{\cris}\hookrightarrow \underline{\End}(\bm{H}_{\cris})$ of $F$-crystals, whose quotient is once again an $F$-crystal over $\mathcal{S}(L)_{\F_p}$.

Moreover, the abelian scheme $A^{\mathrm{KS}}$ extends over a finite \'etale cover of $\mathcal{S}(L)_{(p)}$---namely, the integral canonical model for $\widetilde{\Sh}(L)$---and its endomorphism scheme $\bm{E}$ descends over $\mathcal{S}(L)_{(p)}$, and is equipped with a crystalline realization map $\bm{E} \to \underline{\End}(\bm{H}_{\cris})$.

Given a point $s\to \mathcal{S}(L)_{(p)}$ in characteristic $p$, a \defnword{special endomorphism} over $s$ will be a section of $\bm{E}$ whose crystalline realization lands in the subspace
\[
\bm{L}_{\cris,s}\subset \End(\bm{H}_{\cris,s}).
\]

Given $T\to \mathcal{S}(L)_{(p)}$, a \defnword{special endomorphism} over $T$ will be a section of $\bm{E}$, which induces a special endomorphism at every geometric point of $T$. Write $L(T)$ for the space of special endomorphisms over $T$.

\subsection{}
Consider the de Rham realization $\bm{L}_{\dR}$ of $L$ over $\mathcal{S}(L)_{(p)}$: it is equipped with a non-degenerate quadratic form $\bm{Q}:\bm{L}_{\dR} \to \co_{\mathcal{S}(L)_{(p)}}$, as well as a three-step filtration
\[
0\subset \Fil^1\bm{L}_{\dR}\subset \Fil^0\bm{L}_{\dR}\subset \Fil^{-1}\bm{L}_{\dR} = \bm{L}_{\dR}
\]
by local direct summands. The first step $\Fil^1\bm{L}_{\dR}$ is isotropic of rank $1$ and governs the deformation theory of $\mathcal{S}(L)_{(p)}$. More precisely, let $s\to \mathcal{S}(L)_{(p)}$ be a closed point valued in a perfect field $k$, and let $\widehat{U}_s$ be the formal scheme associated with the complete local ring of $\mathcal{S}(L)_{(p)}$ at $s$. Suppose also that we have a special endomorphism $f\in L(s)$. Then it has a de Rham realization $\bm{f}_{\dR}\in \Fil^0\bm{L}_{\dR,s}$. Let $\widehat{U}_{(s,f)}\to \widehat{U}_s$ be the relative deformation space parameterizing deformations of the pair $(s,f)$.

\begin{proposition}
\label{prop:def theory special end}
Let $k[\epsilon] = k[x]/(x^2)$ be the ring of dual numbers over $k$. Then there is a canonical identification
\[
\widehat{U}_s(k[\epsilon]) = \left\{\text{Isotropic lines $F\subset \bm{L}_{\dR,s}\otimes k[\epsilon]$ lifting $\Fil^1\bm{L}_{\dR,s}$}\right\},
\]
which in turn identifies
\[
\widehat{U}_{(s,f)}(k[\epsilon]) = \left\{F\in\widehat{U}_s(k[\epsilon]):\;\text{$F$ is orthogonal to $\bm{f}_{\dR}$} \right\}
\]
\end{proposition}
\begin{proof}
See Proposition 5.16 of~\cite{mp:reg}.
\end{proof}

\subsection{}\label{subsec:Z Lambda}
Fix a positive definite quadratic lattice $\Lambda$ over $\Z$. Then we have an associated stack $\mathcal{Z}(\Lambda)_{(p)}$, finite over $\mathcal{S}(L)_{(p)}$, and parameterizing, for each $T\to \mathcal{S}(L)_{(p)}$, isometric embeddings $\Lambda\hookrightarrow L(T)$; see Proposition 6.13 of~\cite{mp:reg}.

If $\Lambda$ has rank $1$, then it is isometric to the lattice $\langle m\rangle$ generated by a basis element $e$ satisfying $e^2 = m$, for some positive integer $m\in \Z_{>0}$. In this case, we will write $\mathcal{Z}(m)_{(p)}$ for the associated stack: It parameterizes elements $f\in L(T)$ such that $f\circ f = m\in \bm{E}(T)$. In particular, over $\mathcal{Z}(m)_{(p)}$, we have the universal such endomorphism $\bm{f}\in L(\mathcal{Z}(m)_{(p)})$, along with its de Rham realization
\[
\bm{f}_{\dR}\in H^0(\mathcal{Z}(m)_{(p)},\bm{L}_{\dR}).
\]

Let $\mathcal{Z}^{\mathrm{pr}}(m)_{(p)}\subset \mathcal{Z}(m)_{(p)}$ be the open locus where $\bm{f}_{\dR}$ generates a local direct summand of $\bm{L}_{\dR}$.

The next result records some facts about the above stacks, all of which were shown under the assumption $p\neq 2$ in~\cite{mp:reg}, but whose proofs continue to be valid even without this restriction.

\begin{proposition}
\mbox{}
\label{prop:Z Lambda}
\begin{enumerate}
    \item\label{zlambda:generic}Let $\mathcal{Z}^{\mathrm{sm}}(\Lambda)_{\F_p}\subset \mathcal{Z}(\Lambda)_{(p),\F_p}$ be the complement of the non-smooth locus. Suppose that 
   \[
     \mathrm{rank}(\Lambda) < \mathrm{rank}(L) - 2.
   \]
	Then, for any geometric generic point $\overline{\eta}\to \mathcal{Z}^{\mathrm{sm}}(\Lambda)_{\F_p}$, the tautological map
	\[
    \Lambda \to L(\overline{\eta})
	\]
	is an isomorphism.
	\item\label{zlambda:divisor}For any integer $m\in \Z_{>0}$, the stack $\mathcal{Z}(m)_{(p)}$ is flat over $\Z_{(p)}$. 
	\item\label{zlambda:primitive}Suppose that 
	\[
    \mathrm{rank}(L)\geq 3.
	\]
	Then $\mathcal{Z}^{\mathrm{pr}}(m)_{(p)}$ is normal.
\end{enumerate}
\end{proposition}
\begin{proof}
The first assertion is shown as in Proposition 5.21, the second as in Corollary 5.19, and the third as in Proposition 6.17.
\end{proof}

\subsection{}
In general, for each lattice $L$ as above, we have a unique \emph{normal} integral model $\mathcal{S}(L)$ for $\Sh(L)$ over $\Z$, characterized by the following properties:

\begin{itemize}
	\item If $L_{(p)}$ is self-dual for a prime $p$, then 
	\[
     \mathcal{S}(L)_{\Z_{(p)}} = \mathcal{S}(L)_{(p)}
	\]
	is the integral canonical model for $\Sh(L)$ over $\Z_{(p)}$.
	\item If $L\hookrightarrow L^\diamond$ is as in~\eqref{subsec:diamond}, then the map $\Sh(L)\to \Sh(L^\diamond)$ extends to a finite map $\mathcal{S}(L)\to \mathcal{S}(L^\diamond)$ of $\Z$-stacks.
\end{itemize}

Moreover, for every $T\to \mathcal{S}(L)$, we have a functorially associated space of special endomorphisms $L(T)$ characterized by the following properties:
\begin{itemize}
	\item If $T$ factors through $\Sh(L)$, then $L(T)$ agrees with the space defined in~\eqref{subsec:motives char 0}.
	\item If $L_{(p)}$ is self-dual and $T$ factors through $\mathcal{S}(L)_{(p)}$, then $L(T)$ agrees with the space defined in~\eqref{subsec:self-dual}.
	\item Suppose that $L\hookrightarrow L^\diamond$ and $\Lambda$ are as in~\eqref{subsec:diamond}, and that $L^\diamond(T)$ is the space of special endomorphisms over $T$ viewed as a scheme over $\mathcal{S}(L^\diamond)$. Then there is an isometric embedding
	\[
     \Lambda\hookrightarrow L^\diamond(\mathcal{S}(L)),
	\]
	and a canonical identification
	\[
      L(T) = \Lambda^\perp\subset L^\diamond(T).
	\]
\end{itemize}

For more details on all this, see \S 4 of~\cite{aghmp:colmez}.

\subsection{}
For every prime $\ell$, the $\ell$-adic realization $\bm{L}_{\ell}$ over $\Sh(L)$ extends to an $\ell$-adic sheaf over $\mathcal{S}(L)[\ell^{-1}]$, which we will denote by the same symbol. The $\ell$-adic realizations of any special endomorphism are sections of $\bm{L}_{\ell}$. 

Let $s\to \mathcal{S}(L)$ be a point defined over a field $k(s)$. Fix a separable closure $k(s)^{\mathrm{sep}}$ for $k(s)$, and let $\Gamma_s = \Gal(k(s)^{\mathrm{sep}}/k(s))$ be the associated absolute Galois group. Let $s^{\mathrm{sep}}\to \mathcal{S}(L)$ be the induced $k(s)^{\mathrm{sep}}$-valued point. Then the fiber $\bm{L}_{\ell,s^{\mathrm{sep}}}$ of $\bm{L}_{\ell}$ is a $\Gamma_s$-representation.

For every prime $\ell$ prime to the characteristic of $k(s)$, we have an $\ell$-adic realization map
\begin{equation}\label{eqn:ell adic realization}
L(s)\otimes\Q_{\ell}\to \bm{L}_{\ell,s^{\mathrm{sep}}}^{\Gamma_s}.
\end{equation}

\begin{theorem}\label{thm:tate special}
Suppose that $k(s)$ is finitely generated over its prime field. Then the map~\eqref{eqn:ell adic realization} is an isomorphism.
\end{theorem}
\begin{proof}
As in the proof of Corollary 6.11 of~\cite{mp:tatek3}, we can reduce to the case where $k(s)$ is finite over its prime field. When $k(s)$ is in characteristic $0$, then the result is an immediate consequence of Faltings's isogeny theorem for abelian varieties over number fields. When $k(s)$ is in characteristic $p>0$, then as in Lemma 6.6 of~\cite{mp:tatek3}, we can reduce to the case where $L_{(p)}$ is self-dual and where $L(s)\neq 0$, where the argument used to prove~\cite[Theorem 6.4]{mp:tatek3} works also without the hypothesis $p>2$ to give us the theorem.

Note that in~\cite{mp:tatek3} we required a certain $\ell$-independence condition, which was formulated as Assumption 6.2 of \emph{loc. cit.} However, this condition has been shown to always hold by Kisin~\cite[Corollary 2.3.2]{kisin:lr}.
\end{proof}

\subsection{}
Fix an integer $d>0$, and let $\mathsf{M}_{2d}$ be the moduli stack over $\Z$ of primitively quasi-polarized K3 surfaces of degree $2d$; see~\cite[\S 3]{mp:tatek3}. 

Let $(\bm{\mathcal{X}},\bm{\xi})\to \mathsf{M}_{2d}$ be the universal quasi-polarized K3 surface. For each prime $\ell$, we have its $\ell$-adic primitive cohomology sheaf $\bm{P}^2_{\ell}$ over $\mathsf{M}_{2d}[\ell^{-1}]$. For each prime $p$, we also have the associated $F$-crystal $\bm{P}^2_{\cris}$ over $\mathsf{M}_{2d,\F_p}$. Finally, we have the filtered vector bundle $\bm{P}^2_{\dR}$ over $\mathsf{M}_{2d}$ obtained from the primitive relative de Rham cohomology of $\mathcal{\bm{X}}$.

We also have the full de Rham cohomology $\bm{H}^2_{\dR}$ of $\bm{\mathcal{X}}\to \mathsf{M}_{2d}$. The canonical Poincar\'e pairing on it is induced from a quadratic form
\[
\bm{Q}:\bm{H}^2_{\dR}\to \Reg{\mathsf{M}_{2d}}.
\]
This is obvious when $2$ is invertible, and follows from Theorem 4.7 of~\cite{Ogus1983-xq} in general. 

\subsection{}\label{subsec:k3 deform}
There is a three step filtration
\[
0\subset \Fil^2\bm{H}^2_{\dR}\subset \Fil^1\bm{H}^2_{\dR}\subset \Fil^0\bm{H}^2_{\dR} = \bm{H}^2_{\dR}
\]
by local direct summands, where $\Fil^2\bm{H}^2_{\dR}$ is isotropic of rank $1$, and governs the deformation theory of $\mathsf{M}_{2d}$. 

More precisely, if $\mathrm{ch}_{\dR}(\bm{\xi})\in H^0(\mathsf{M}_{2d},\Fil^1\bm{H}^2_{\dR})$ is the de Rham Chern class of $\bm{\xi}$, we have 
\[
\Fil^2\bm{H}^2_{\dR}\subset \langle \mathrm{ch}_{\dR}(\bm{\xi})\rangle^{\perp} = \bm{P}^2_{\dR}\subset \bm{H}^2_{\dR}.
\]
Moreover, if $s\to \mathsf{M}_{2d}$ is a point valued in a perfect field $k$, then the set of lifts of $s$ to a $k[\epsilon]$-valued point of $\mathsf{M}_{2d}$ is canonically identified with the set of isotropic lines $F\subset\bm{P}^2_{\dR,s}\otimes k[\epsilon]$ that lift $\Fil^2\bm{H}^2_{\dR,s}$.

\subsection{}
As in~\cite[\S 3.10]{mp:tatek3}, let $N$ be the self-dual lattice $U^{\oplus 3} \oplus E_8^{\oplus 2}$, where $U$ is the hyperbolic plane. Choose a hyperbolic basis $e,f$ for the first copy of $U$, and set
\[
L_d = \langle e-df\rangle^{\perp}\subset N.
\]

There exists a canonical $2$-fold \'etale cover $\tilde{\mathsf{M}}_{2d}\to \mathsf{M}_{2d}$, whose restriction to $\mathsf{M}_{2d}[\ell^{-1}]$ parameterizes isometric trivializations of the rank $1$ sheaf $\det(L_d)\xrightarrow{\simeq}\Z_\ell\xrightarrow{\simeq}\underline{\det}(\bm{P}^2_{\ell})$.

There is now a canonical period map (see Corollary 5.4 of~\cite{mp:tatek3})
\[
\iota^{\mathrm{KS}}_{\Q}:\; \tilde{\mathsf{M}}_{2d,\Q}\to \Sh(L_d).
\]
Moreover, for $? = \dR,\ell$, we have a canonical isometry (see Proposition 4.6 of~\cite{mp:tatek3}):
\[
\alpha_?:\bm{L}_{?}(-1)\vert_{\tilde{\mathsf{M}}_{2d,\Q}}\xrightarrow{\simeq}\bm{P}^2_?\vert_{\tilde{\mathsf{M}}_{2d,\Q}}.
\]

\begin{proposition}
\label{prop:integral kuga satake}
Let $\tilde{\mathsf{M}}^{\mathrm{sm}}_{2d}\subset \tilde{\mathsf{M}}_{2d}$ be the complement of the non-smooth loci in all its special fibers. Then $\iota^{\mathrm{KS}}_{\Q}$ extends to an \'etale morphism
\[
\iota^{\mathrm{KS}}:\;\tilde{\mathsf{M}}^{\mathrm{sm}}_{2d}\to \mathcal{S}(L_d).
\] 
\end{proposition}
\begin{proof}
It is enough to prove the proposition over $\Z_{(p)}$ for a fixed prime $p$. Let $L_d\hookrightarrow L^\diamond$ be an isometric embedding as in~\eqref{subsec:diamond} with $L^\diamond_{(p)}$ a self-dual lattice. In fact, it is not hard to see, using the classification of quadratic spaces over $\Q$ using the Hasse principle, that we can choose $L^\diamond$ so that $L^\diamond_{\Z_p}$ is isometric to $N_{\Z_p}$, so that $\Lambda = L_d^\perp\subset L^\diamond$ has rank $1$.

Now, $\mathcal{S}(L_d)\to \mathcal{S}(L^\diamond)$ is the normalization of $\mathcal{S}(L^\diamond)$ in $\Sh(L_d)$, and so it is enough to show that the composition
\[
\tilde{\mathsf{M}}_{2d,\Q} \xrightarrow{\iota^{\mathrm{KS}}_{\Q}}\Sh(L_d)\to \Sh(L^\diamond)
\]
extends to a map $\tilde{\mathsf{M}}^{\mathrm{sm}}_{2d,\Z_{(p)}}\to \mathcal{S}(L^\diamond)_{{(p)}}$. This is shown just as in the proof of~\cite[Proposition 4.7]{mp:tatek3}, using the fact that $\mathcal{S}(L^\diamond)_{{(p)}}$ is an integral canonical model for its generic fiber.

It remains to show that the map is \'etale. Suppose that $\Lambda = \langle m\rangle$ for some integer $m\in \Z_{>0}$, and as in~\eqref{subsec:Z Lambda}, consider the finite morphism
\[
\mathcal{Z}(m)_{(p)}\to \mathcal{S}(L^\diamond)_{{(p)}}
\]
parameterizing $f\in L^\diamond(T)$ with $f\circ f = m\in \bm{E}(T)$. 

The canonical embedding $\Lambda\hookrightarrow L^\diamond(\mathcal{S}(L))$ determines a finite morphism
\[
\mathcal{S}(L)_{\Z_{(p)}} \to \mathcal{Z}(m)_{(p)},
\] 
which is an open and closed immersion in the generic fiber; see Lemma 7.1 of~\cite{mp:reg}. 

Now, it is enough to show that the induced map 
\begin{equation}\label{eqn:period to Z}
\tilde{\mathsf{M}}^{\mathrm{sm}}_{2d,\Z_{(p)}}\to \mathcal{Z}(m)_{(p)}
\end{equation}
is \'etale.


Arguing as in the proof of Lemma 6.16(4) of~\cite{mp:reg}, we first find that~\eqref{eqn:period to Z} factors through the open substack $\mathcal{Z}^{\mathrm{pr}}(m)_{(p)}$.

Over $\mathcal{Z}^{\mathrm{pr}}(m)_{(p)}$, the de Rham realization $\bm{f}_{\dR}$ of the tautological element $f\in L^\diamond(\mathcal{Z}(m)_{(p)})$ spans a local direct summand of $\bm{L}^\diamond_{\dR}\vert_{\mathcal{Z}^{\mathrm{pr}}(m)_{(p)}}$. Its orthogonal complement gives us a vector sub-bundle
\[
\bm{L}_{\dR} \coloneqq \langle \bm{f}_{\dR}\rangle^{\perp}\subset \bm{L}^\diamond_{\dR}\vert_{\mathcal{Z}^{\mathrm{pr}}(m)_{(p)}}
\]
over $\mathcal{Z}^{\mathrm{pr}}(m)_{(p)}$, whose restriction over $\Sh(L)$ is canonically identified with $\bm{L}_{\dR,\Q}$. Moreover, the isotropic line $\Fil^1\bm{L}^\diamond_{\dR}\subset \bm{L}^\diamond_{\dR}$ is orthogonal to $\bm{f}_{\dR}$ over $\mathcal{Z}(m)_{(p)}$, and so gives us a local direct summand
\[
\Fil^1\bm{L}_{\dR}\subset \bm{L}_{\dR}
\]
over $\mathcal{Z}^{\mathrm{pr}}(m)_{(p)}$.

To show that~\eqref{eqn:period to Z} is \'etale, it suffices, using our knowledge of the deformation theory of $\mathcal{Z}^{\mathrm{pr}}(m)_{(p)}$ (see Proposition~\ref{prop:def theory special end}), and that of $\mathsf{M}_{2d}$ (see~\eqref{subsec:k3 deform}), to show that the isometry $\alpha_{\dR}$ in the generic fiber extends to an isometry
\[
\bm{L}_{\dR}(-1)\vert_{\tilde{\mathsf{M}}^{\mathrm{sm}}_{2d,\Z_{(p)}}}\xrightarrow{\simeq}\bm{P}^2_{\dR}
\]
over $\tilde{\mathsf{M}}^{\mathrm{sm}}_{2d,\Z_{(p)}}$. See the proof of Theorem 5.8 of~\cite{mp:tatek3} for an explanation of this.

To show that $\alpha_{\dR}$ extends, it is enough to do so over the ordinary locus of $\tilde{\mathsf{M}}^{\mathrm{sm}}_{2d,\Z_{(p)}}$; see Proposition 5.11 of~\cite{mp:tatek3}. The extension over the ordinary locus is accomplished exactly as in Lemma 5.10 of the same article.
\end{proof}

\subsection{}
Since $\tilde{\mathsf{M}}^{\mathrm{sm}}_{2d}$ is normal, the isometry $\alpha_{\ell}$ over the generic fiber extends:
\[
\alpha_{\ell}:\bm{L}_{\ell}(-1)\vert_{\tilde{\mathsf{M}}^{\mathrm{sm}}_{2d}[\ell^{-1}]}\xrightarrow{\simeq}\bm{P}^2_\ell\vert_{\tilde{\mathsf{M}}^{\mathrm{sm}}_{2d}[\ell^{-1}]}.
\]

\begin{proposition}
\label{prop:special end to picard}
For every point $s\to \tilde{\mathsf{M}}^{\mathrm{sm}}_{2d}$, we have a canonical isometry
\[
\mathrm{Pic}(\bm{\mathcal{X}}_s)\supset \langle \bm{\xi}_s\rangle^{\perp}\xrightarrow{\simeq}L(\iota^{\mathrm{KS}}(s))
\]
compatible with $\ell$-adic realizations on both sides via the isometry $\alpha_{\ell}$, for any prime $\ell$ prime to the characteristic of $k(s)$.
\end{proposition}
\begin{proof}
As in the proof of~\cite[Theorem 4.17]{mp:tatek3}, it is enough to show two assertions:
\begin{itemize}
	\item For every $f\in L(\iota^{\mathrm{KS}}(s))$, the deformation space of the triple $(\bm{\mathcal{X}}_s,\bm{\xi}_s,f)$ admits a component that is flat over $\Z_p$.
	\item For every $\eta\in \langle \bm{\chi}_s\rangle^{\perp}\subset \mathrm{Pic}(\bm{\mathcal{X}}_s)$, the deformation space of the triple $(\bm{\mathcal{X}}_s,\bm{\xi}_s,\eta)$ admits a flat component.
\end{itemize}

The second assertion follows from a result of Lieblich-Olsson~\cite[A.1]{Lieblich2011-ww}. As for the first, since $\iota^{\mathrm{KS}}$ is \'etale, it is enough to show that the deformation space of $f$ over $\mathcal{S}(L_d)$ is flat. In fact, as in the proof of Proposition~\ref{prop:integral kuga satake}, we can assume that we are working at a smooth point of the stack $\mathcal{Z}^{\mathrm{pr}}(m)_{(p)}$. In this case, the assertion follows as in Proposition 7.18 of~\cite{mp:reg}, using~\ref{zlambda:generic} as input.
\end{proof}

\begin{proof}[Proof of Theorem~\ref{thm:k3}]
Suppose that $X$ is a K3 surface over a finitely generated field $k$. Fix a separable closure $k^{\mathrm{sep}}$ of $k$, and let $\Gamma_k = \Gal(k^{\mathrm{sep}}/k)$ be the associated absolute Galois group. We have to show that, for all $\ell$ prime to the characteristic of $k$, the realization map
\[
\mathrm{Pic}(X) \to H^2_{\et}(X_{k^{\mathrm{sep}}},\Q_\ell)(1)^{\Gamma_s}
\]
is an isomorphism.

After replacing $k$ with a finite, separable extension, we can assume that $X$ admits a quasi-polarization $\chi$ of degree $2d$ for some integer $d$, and thus is associated with a $k$-valued point $s\to \mathsf{M}_{2d}$. We can assume that $s$ admits a lift to $\tilde{\mathsf{M}}_{2d}$, which we will once again denote by $s$.

If $s$ factors through $\tilde{\mathsf{M}}^{\mathrm{sm}}_{2d}$, then we are done by Theorem~\ref{thm:tate special} and Proposition~\ref{prop:special end to picard}. 

Otherwise, $X$ is a superspecial K3 surface: That is, the de Rham Chern class $\mathrm{ch}_{\dR}(\xi)$ lies in $\Fil^2H^2_{\dR}(X/k)$; see~\cite[2.2]{Ogus1979-cv}. Now, since there are no non-constant families of quasi-polarized superspecial K3 surfaces (see~\cite[Remark 2.7]{Ogus1979-cv}), and since every irreducible locus of the supersingular locus of $\mathsf{M}_{2d,\F_p}$ has dimension $9$ (see~\cite[Theorem 15]{Ogus2001-wy}), the theorem in general follows from the validity of the Tate conjecture for points in $\tilde{\mathsf{M}}^{\mathrm{sm}}_{2d,\F_p}$, and Artin's result on the constancy of Picard rank in families of supersingular K3 surfaces~\cite[Corollary (1.3)]{Artin1974-uh}.
\end{proof}

\begin{remark}
Given the Tate conjecture, one should be able to show that every superspecial K3 surface is actually the Kummer surface associated with a superspecial abelian surface, using which one should be able to extend the period map $\iota^{\mathrm{KS}}$ over all of $\tilde{\mathsf{M}}_{2d}$.
\end{remark}

\printbibliography

\end{document}